\documentclass[10pt]{amsart}

\usepackage{graphicx, yfonts, wasysym, appendix, url, verbatim}
\usepackage[margin=4cm]{geometry}
\usepackage{rotating}
\usepackage{amsbsy}
\usepackage{graphicx}
\usepackage{ccaption}
\usepackage{xcolor}
\usepackage{times}
\usepackage{wrapfig}
\usepackage{nicefrac}
\usepackage{amsmath, amssymb, amsthm} 

\usepackage{hyperref}

\newcommand{\pD}[2]{\frac{\partial #1}{\partial #2}}

\newcommand{\vn}[1]{\lVert#1\rVert}
\newcommand{\IP}[2]{\left< #1 , #2 \right>}

\newcommand{\rD}[2]{\frac{d #1}{d #2}}

\newcommand{\sfrac}[2]{\text{\fontsize{5}{5}\selectfont$\frac{#1}{#2}$}}
\newcommand\numberthis{\addtocounter{equation}{1}\tag{\theequation}}
\newcommand{\Tmin}{T_{\text{min}}}
\newcommand{\Tm}{T_{\text{max}}}

\newcommand{\R}{\ensuremath{\mathbb{R}}}

\newcommand{\FP}{\mathcal{O}}
\newcommand{\N}{\ensuremath{\mathbb{N}}}
\renewcommand{\S}{\ensuremath{\mathbb{S}}}

\allowdisplaybreaks[4]

\newtheorem{thm}{Theorem}[section]
\newtheorem{cor}[thm]{Corollary}
\newtheorem{prop}[thm]{Proposition}
\newtheorem{lem}[thm]{Lemma}

\theoremstyle{remark}
\newtheorem*{rmk}{Remark}


\title{On a curvature flow model for embryonic epidermal wound healing}
\author{Shuhui He, Glen Wheeler, Valentina-Mira Wheeler}

\address{Shuhui He \\
           Institute for Mathematics and its Applications \\
           University of Wollongong\\
           Northfields Avenue\\
           Wollongong, NSW, 2522, Australia\\
           email: sh807@uowmail.edu.au
           }
\address{Glen Wheeler \\
           Institute for Mathematics and its Applications \\
           University of Wollongong\\
           Northfields Avenue\\
           Wollongong, NSW, 2522, Australia\\
           email: glenw@uow.edu.au
           }
\address{Valentina-Mira Wheeler \\
           Institute for Mathematics and its Applications \\
           University of Wollongong\\
           Northfields Avenue\\
           Wollongong, NSW, 2522, Australia\\
           email: vwheeler@uow.edu.au
           }

\begin{document}

\begin{abstract}

The paper studies a curvature flow linked to the physical phenomenon of wound
closure. Under the flow we show that a closed, initially convex or
close-to-convex curve shrinks to a round point in finite time. We also study
the singularity, showing that the singularity profile after continuous
rescaling is that of a circle.
We additionally give a maximal time estimate, with an application to the
classification of blowups.
\end{abstract}

\maketitle

\tableofcontents

\section{Introduction}
Wound healing is a complex and essential biological process for an organism to
repair damaged tissue and therefore survive the hostile external environment.
Ideally, the injury site is to be replaced by tissue that has its original
structure and functions. This is so called regenerative wound healing. Early
research, such as~\cite{colwell2003fetal,GCL2007,GWB2008}, show that some
eukaryotic organisms are able to perform regenerative healing throughout their
life time, while humans only have this ability during prenatal development. In
adulthood, wound healing typically generates a mass of non-functioning cells
and structure which is referred to as scarring. This limitation of adult wound
healing ability may lead to severe clinical consequence such as non-healing
wounds, congestive heart failure and liver cirrhosis.
Various existing models for wound healing focus on different aspects of the
complicated adult wound healing process, see for
example~\cite{dale1996mathematical, mcdougall2006fibroblast, OLS1997,
ravasio2015gap}.

Instead of studying the complicated biomechanical process, we focus on a much
simpler setting: embryonic epidermal wound healing. This healing process is
regenerative and has also been extensively studied. A current modeling technique
is to study the change of shape of the wound under prescribed forces
acting on the leading edge of the wound. We adapt the approach suggested
in~\cite{almeida2008tissue, ravasio2015gap}, to describe the movement of the
leading edge with a curvature term and a source term.

In our model, we treat the leading edge of a wound as a simple closed plane
curve. This assumption is valid as in the embryonic wound healing setting, an
epidermal wound typically contains very few layers of cells, and can be
considered as a flat surface. We investigate the forces that contribute to the
movement of the curve.

There are three forces acting on the leading edge that we take into
consideration. The first force is actin cable contraction. When a wound is
presented, actin is assembled to surround the wound opening to form a dense
actin cable network. The contraction of the actin cable acts as a purse-string
and contributes to a local force proportional to the curvature of the leading edge.
The second force comes from lamellipodial crawling. This is a biological
response that cells on the leading edge extend protrusions into
tissues within the wound and drag themselves forward, advancing into the wound. This
closing force is locally constant and is in the direction normal to the edge.
The last force is epidermal tension, a force that comes from the pulling of
surrounding cells, resisting the closure of the wound. This force is again
acting in the normal direction and can be considered locally constant.

The velocity $V$ at a point on the leading edge is therefore
\begin{equation}
\label{EQvelocity}
V(x) = \sigma_1 k(x)\nu(x) + (\sigma_2^L - \sigma_2^E)\nu(x)
     = \sigma_1 k(x)\nu(x) + \sigma_2\nu(x)
\end{equation}
where $\sigma_1, \sigma_2^L, \sigma_2^E$ are constants corresponding to the
three forces above, with $\sigma_1 > 0$, $\sigma_2^L, \sigma_2^E \in \R$,  $x$
is a position along the leading edge, $k(x)$ is the curvature of the leading
edge at $x$ (it is positive for convex wounds), $\nu(x)$ is the inward-pointing
unit normal to the leading edge at $x$, and $\sigma_2 := \sigma_2^L -
\sigma_2^E$.

Throughout this article we assume that:
\begin{enumerate}
\item[($\sigma_1 > 0$)] Physically this means that convex regions of the actin
cable tend to contract inward, and concave regions tend to relax outward.  This
is reasonable, as elliptical regions are convex and if $\sigma_1 < 0$ then they
would expand, which is not physical.  Mathematically this is required for
parabolicity and to therefore generate a unique solution from given initial
data.
\item[($\sigma_2 > 0$)] This is a physical assumption amounting to the claim
that lamellipodial crawling is a greater force than epidermal tension.
Physically this is reasonable, for two reasons: One, a locally straight region
of a closed wound (where $k=0$) moves inward, if $\sigma_2 < 0$, then such
regions would move outward; and two, if $\sigma_2 < 0$, then a circular wound
of radius $-\frac{\sigma_1}{\sigma_2}$ would remain stationary under the flow
and never heal.
\end{enumerate}

\begin{rmk}
While the above considerations justify the \emph{sign} of $\sigma_1$ and
$\sigma_2$, they do not address the fact that here we assume they are
\emph{constant}.  This assumption is one of simplicity.
In later work we will replace this with anisotropy and other more general
hypotheses.
\end{rmk}

The velocity \eqref{EQvelocity} gives rise to an evolution equation for  
embedded closed plane curves, a curvature flow, that is a modified version of
the classical curve shortening flow.  The curve shortening flow has velocity
given by \eqref{EQvelocity} with $\sigma_1 = 1$ and $\sigma_2 = 0$.
The curve shortening flow has been extensively studied, see for
example \cite{gage1986heat,grayson1987heat,huisken1984flow}.

Let us now describe the mathematical model and state our main result.
Let $\gamma_0:\S^1\to\R^2$ be a plane embedded, closed curve describing
an initial wound.  Consider a family of closed, embedded curves $\gamma:
\S^1\times[0,T)\to\R^2$ representing the leading edge of a wound as it evolves
in time, according to
\begin{align}
\partial_t\gamma(x, t) &= F(x, t)\nu = \left(\sigma_1 k(x,t) + \sigma_2\right)\nu(x,t) \qquad\text{in $\S^1\times[0,T)$,}\label{C-E problem}\\
\gamma(\cdot,0) &= \gamma_0(\cdot),\notag
\end{align}
The problem \eqref{C-E problem} is a system of degenerate nonlinear parabolic PDE
of second order on a compact domain. Local well-posedness for sufficiently smooth data follows by standard 
techniques.

In this paper, our focus is on global analysis for the flow.
Our main result is as follows.
\begin{thm}
Let $\gamma: \S^1\times[0,T)\to\R^2$ be a family of plane, smooth, embedded, closed
curves evolving under the flow \eqref{C-E problem}.
Then the flow exists for at most finite time $T<\infty$.
If the initial curve $\gamma_0$ is convex,
then the curves $\gamma_t(\cdot) := \gamma(\cdot,t)$ contract exponentially
fast in the smooth topology to a smooth round point as $t\nearrow T$.

Furthermore, if the initial curve $\gamma_0$ satisfies for some $\alpha\in(0,2)$
\begin{equation}
\label{nonconvexbound}
L_0\vn{k_{s}}_2^2\Big|_{t=0}
\le 
   \frac{1}{196\sigma_1^2}\left(
   \sqrt{
	25\sigma_2^2
	+ \frac{14\sigma_1(2-\alpha)}{\Tm}
	}
 - 5\sigma_2
	\right)^2
\,,
\end{equation}
where $\Tm$ is an upper bound for the maximal time $T$ of
smooth existence that depends only on $\gamma_0$ (see Theorem
\ref{TMmaximaltimeestimates}), then there exists a $t_0$ such that for all
$t\in(t_0,T)$, $\gamma(\cdot, t)$ is convex, and we have again smooth
convergence to a round point as $t\nearrow T$.
\label{maintheorem}
\end{thm}

Our method of proof is inspired by the existing literature.
We use the integral estimate method of Gage-Hamilton \cite{gage1986heat} in
Section 3 to show, essentially, that the flow may be smoothly extended so long
as the enclosed area is bounded away from zero.
This implies that the enclosed area must vanish at final time.
An additional argument and convexity is required to show that length also
vanishes at final time -- a key step in arguments to come.
The main technical difficulty is to identify the asymptotic shape of the flow.
Indeed, there are conjectures \cite{dallaston2016curve} on non-local flows of a similar form that
suspect the asymptotic shape is an ellipse or something more exotic.

A modified curve shortening flow with anisotropy is studied in
\cite{chou2001curve,chouzhu1}.  There, it is shown that some classes of curve
shortening flows shrink convex curves to round points.
In a later article the non-convex case is studied \cite{chouzhu2}, however the
condition (0.7) they place there rules out $\sigma_2 = const$ unless $const =
0$.
We use the ideas of Chou-Zhu for various estimates, especially for the a-priori curvature estimate on the rescaled flow.
These appear in the Appendix.
We thank the authors for the inspiration.
We should also note that the Chou-Zhu curvature estimate uses in turn the ideas
of Gage-Hamilton \cite{gage1986heat}, in particular, using an estimate for the
entropy to (eventually) bound the curvature. 

We perform a natural rescaling that sets the \emph{final} value of the enclosed area to $\sigma_1\pi$.
The aforementioned curvature estimate allows us to extract a smooth limit from the rescaling.
In order to identify the shape of this limit we use the monotonicity formula
from Huisken \cite{huisken1990asymptotic}, with a slight modification in order
to apply it to our setting here.

Our partial result on the non-convex case is via integral estimate techniques.
In Section 5, our key idea is to prove that a flow that is initially
almost-convex can not admit a continuous rescaling (as in Section 5) that remains
non-convex for all rescaled time.
Then, once the flow is convex, arguments from Sections 2--4 and Appendices A, B in the paper apply
to give smooth convergence of the rescaled flow to a round circle.
Once written in terms of the original flow, the smallness condition relies upon
an upper bound for maximal time of existence.  This is elementary to derive,
using either the avoidance principle for solutions to the flow, or the
evolution of length (or area).

Much harder than the upper bound for maximal time is a lower bound.
This has been classically interesting in the literature, and has impact on
certain discrete rescalings around singularities, called \emph{blowups} (see
Theorem \ref{TMlifespanappl} below).
For this we use a concentration-compactness alternative pioneered by Struwe \cite{struwe}.
Differences here abound: We use a product functional (length and curvature in
$L^2$) which is scale invariant, the product functional may not be globally
small regardless of initial data, and our flow can not exist globally.
These features necessitate changes in the proof of the
concentration-compactness alternative and the lower bound on maximal time.
Key parts of these arguments are inspired by methods used by the second author
in the study of curvature flow of higher-order and with the third author on
curvature flow with free boundary \cite{w1,w2,w3,w4,w5,w6}.
The maximal time estimate is:

\begin{thm}[Lifespan theorem]
\label{TMmaximaltimeestimates}
Let $\gamma:\S^1\times[0,T)\to\R^2$ be a non-convex solution to \eqref{C-E problem}.
	There are constants $\rho\in(0,1)$, $\varepsilon_1	>0$, and $c_0<\infty$ such that
\begin{equation*}
	\sup_{x\in\R^2}L_{B_\rho(x)}\int_{\gamma^{-1}(B_\rho(x))} k^2\,ds\Big|_{t=0} = \varepsilon(x) \le \varepsilon_1
\end{equation*}
implies that the maximal time $T$ satisfies
\begin{equation*}
T_{\text{max}} := \frac{1}{2\pi\sigma_1\sigma_2}\min\{L_0\sigma_1, A_0\sigma_2\}
 \ge T \ge \frac{1}{c_0}\rho^2 := \Tmin\,,
\end{equation*}
where $L_0$ and $A_0$ denote the initial length and enclosed area of the flow.
We additionally have the estimate
\begin{equation}
\label{EQltest}
	L_{B_{\frac{\rho}{2}}(x)}\int_{\gamma^{-1}(B_\frac\rho2(x))}k^2\,ds
	\le 
	c\varepsilon_1
\qquad\text{ for }\qquad
0\le t \le \Tmin
\,.
\end{equation}
\end{thm}

Qualitatively, the lifespan theorem implies that a certain quantum of curvature
concentrates along the flow in smaller and smaller intervals as a singularity
develops.
This has been used by, for example, Kuwert-Sch\"atzle in their study of the
Willmore flow of surfaces \cite{kuwert,kuwert2}.
In Section 7 we describe how it can be used to partially classify blowups of
non-circular singular points for our flow here.

\begin{thm}[Non-circular blowups]
\label{TMlifespanappl}
Let $\gamma:\S^1\times[0,T)\rightarrow\R^2$ be a family of plane, smooth, immersed closed curves evolving under the flow \eqref{C-E problem}.
Suppose that $\gamma_t$ do not contract to a round point as $t\nearrow T$.
Then:
\begin{itemize}
\item $\gamma_t$ is \emph{not convex} for every $t\in[0,T)$;
\item For all $t\in[0,T)$,
\begin{equation*}
L\vn{k_{s}}_2^2\big(t\big)
>
   \frac{1}{196\sigma_1^2}\left(
   \sqrt{
	25\sigma_2^2
	+ \frac{14\sigma_1(2-\alpha)}{\Tm}
	}
 - 5\sigma_2
	\right)^2\,;
\end{equation*}
\item The discrete blowup $\gamma^\infty$ exists, and has the following properties:
\begin{itemize}
\item is an \emph{ancient non-convex} solution to curve shortening flow, and
\item each component of the blowup is either non-compact, non-embedded, or both.
\end{itemize}
\end{itemize}
\end{thm}

We conjecture that for some values of $\sigma_1$ and $\sigma_2$ there is non-preservation of embeddedness and non-round singular profiles.
It is natural to guess that these singularities are both non-compact and non-embedded in the blowup.


\section{Evolutions of length, area and curvature}
\label{sec:C-E evolution}

In this section and the rest of this article, latin subscripts are used to
indicate partial derivatives unless otherwise stated.

Let $\gamma:\S^1\times[0,T)\to\R^2$ be a smooth family of closed, embedded,
plane curves moving by the flow \eqref{C-E problem}. The time derivative in
\eqref{C-E problem} is taken along fixed values of the natural parameter $u\in\S^1$.
In order to study the intrinsic properties of the family of curves efficiently,
we reparametrise the curves by arc-length, and denote the arc-length parameter
by $s$. Note that the arc-length and the time derivatives do not commute. Indeed, we have
$\partial_s = |\gamma_u|^{-1}\partial_u$, and
$|\gamma_u|^{-1}:\S^1\times[0,T)\rightarrow\R$ is not in general (in this
setting) constant in $t$.

Let $\tau(u,t) = \frac{\gamma_u(u,t)}{|\gamma_u(u,t)|}$ be the unit tangent
vector to $\gamma(u,t)$, and $\nu$ the inward pointing unit
normal at $\gamma(u,t)$. By possibly changing orientation (mapping $u\mapsto-u$) we may assume that
\[
\nu(u,t) = \text{rot}_{\frac\pi2}\tau(u,t)
\]
where $\text{rot}_\frac\pi2(x,y) = (-y,x)$ is counter-clockwise rotation through $\frac\pi2$ radians.

The well-known Frenet-Serret equations are
\[
\partial_s\tau = k\nu, \qquad \partial_s\nu = -k\tau.
\]
We wish to compute the evolution of geometric quantities under the flow
\eqref{C-E problem}. Let us first derive the commutator for $\partial_t$ and $\partial_s$.

\begin{lem}\label{C-E commutator}
Let $\gamma:\S^1\times[0,T)\to\R^2$ be a solution to \eqref{C-E problem}. The differential operators with respect to arc-length and time satisfy
\[
\partial_t\partial_s -\partial_s\partial_t = kF\partial_s =k(\sigma_1k+\sigma_2)\partial_s.
\]
\end{lem}
\begin{proof}
Using the Frenet-Serret equations, we compute
\begin{align*}
2|\gamma_u||\gamma_u|_t = \partial_t|\gamma_u|^2 &= 2\IP{\gamma_{ut}}{\gamma_u}
= 2\IP{(F\nu)_u}{\gamma_u}
\\
&= 2|\gamma_u|^2 \IP{F_s\nu+F\nu_s}{\gamma_s}
\\
&= 2|\gamma_u|^2 \IP{-kF\tau}{\tau}
\\
&= -2kF|\gamma_u|^2.
\end{align*}
Hence 
\begin{equation}
|\gamma_u|_t = -kF|\gamma_u|,
\end{equation}
and
\begin{align*}
\partial_t\partial_s &= \partial_t\left(|\gamma_u|^{-1}\partial_u\right) = |\gamma_u|^{-1}\partial_u\partial_t-|\gamma_u|^{-2}|\gamma_u|_t\,\partial_u
\\
&= \partial_s\partial_t+kF\partial_s = \partial_s\partial_t+k(\sigma_1k+\sigma_2)\partial_s.
\end{align*}
\end{proof}

Let us compute the evolution of the unit tangent and unit normal with the commutator.

\begin{lem}\label{tg nml evolution}
Let $\gamma:\S^1\times[0,T)\to\R^2$ be a solution to \eqref{C-E problem}. We have the following evolution equations
\begin{align*}
\tau_t &= F_s\nu = \sigma_1k_s\nu,
\\
\nu_t &= -F_s\tau = -\sigma_1k_s\tau.
\end{align*}
\end{lem}
\begin{proof}
Applying Lemma~\ref{C-E commutator} and the Frenet-Serret equations we obtain
\begin{align*}
\tau_t &= (\gamma_s)_t = \gamma_{ts}+kF\gamma_s
= (F\nu)_s+kF\tau
\\
&= F_s\nu-Fk\tau +kF\tau
= F_s\nu. 
\end{align*}
Since $\gamma$ is a plane curve, we have $\IP{\tau}{\nu}=0$ and $\nu_t$ must be parallel to $\tau$. We have
\begin{align*}
\nu_t = \IP{\tau}{\nu_t}\tau = -\IP{\tau_t}{\nu}\tau = -F_s\tau.
\end{align*}
\end{proof}

We may use the above lemmata to compute the evolution of length and the enclosed area of the curve.

\begin{prop}[Length evolution]\label{lenghtevolution}
Let $\gamma:\S^1\times[0,T)\to\R^2$ be a smooth family of curves evolving under the flow \eqref{C-E problem}. The length $L$ of the curve $\gamma$ evolves according to
\begin{equation*}
L'(\gamma(\cdot,t)) = -\int_\gamma kF\,ds = -\sigma_1\int_\gamma k^2 \,ds - 2\pi\omega\sigma_2\,,
\end{equation*}
where $\omega = \int_\gamma k\,ds$ is the winding number of $\gamma$.
\end{prop}
\begin{proof}
Let us first note that 
\[
\partial_t\,ds = \partial_t(|\gamma_u\,|du) = |\gamma_u|_t\,du=-kF\,ds.
\]
By definition of the length of a curve and the fact that $k = \theta_s$ where $\theta$ is the angle made by $\tau$ and a fixed vector, we compute
\begin{align*}
L'(\gamma(\cdot,t)) &= 
\int_\gamma |\gamma_u|_t \,du 
= -\int_\gamma kF ds
\\
&= -\int_\gamma k(\sigma_1 k +\sigma_2)\,ds
= -\sigma_1\int_\gamma k^2 \,ds - 2\pi\omega\sigma_2\,.
\end{align*}
\end{proof}

\begin{lem}[Area evolution]
Let $\gamma:\S^1\times[0,T)\to\R^2$ be a smooth family of curves evolving under the flow \eqref{C-E problem}.
The signed enclosed area $A$ of the curve $\gamma$ evolves according to
\begin{align*}
A'(\gamma(\cdot,t)) = -\int_\gamma F\,ds = -2\pi\omega\sigma_1 -\sigma_2 L\,.
\end{align*}
\label{areaevolution}
\end{lem}
\begin{proof}
Note that the definition of signed enclosed area is
\[
A(\gamma(\cdot,t)) = -\frac{1}{2}\int_\gamma \IP{\gamma}{\nu}\,ds\,.
\]
Therefore, we see the rate of change of the area satisfies
\begin{align*}
A'(\gamma(\cdot,t)) &= \partial_t\left(-\frac{1}{2}\int_\gamma \IP{\gamma}{\nu}\,ds\right)
= -\frac{1}{2}\int_\gamma \IP{\gamma_t}{\nu}+\IP{\gamma}{\nu_t} -\IP{\gamma}{\nu}kF\,ds
\\
&= -\frac{1}{2}\int_\gamma \IP{F\nu}{\nu}+\IP{\gamma}{-F_s\tau} -\IP{\gamma}{F\tau_s}\,ds
\\
&= -\frac{1}{2}\int_\gamma F +\IP{\gamma}{-(F\tau)_s}\,ds
 = -\frac{1}{2}\int_\gamma F\,ds -\frac{1}{2}\int_\gamma \IP{\gamma_s}{F\tau} \,ds
\\
&= -\int_\gamma F\,ds = -\int_\gamma \sigma_1k+\sigma_2\,ds =
-2\pi\omega\sigma_1 -\sigma_2 L(\gamma(\cdot,t))\,.
\end{align*}
\end{proof}

\newcommand{\prea}{Let $\gamma:\S^1\times[0,T)\to\R^2$ be a smooth family of curves evolving under the flow \eqref{C-E problem}. }

\begin{lem}[Energy for the flow]
	\label{LMenergy}\prea
Define the functional
\[
E(\gamma(\cdot,t)) = \sigma_1 L(\gamma(\cdot,t))
                   + \sigma_2 A(\gamma(\cdot,t))\,.
\]
Then
\[
E'(\gamma(\cdot,t)) = - \int_\gamma F^2\,ds\,.
\]
\end{lem}
\begin{proof}
We calculate
\begin{align*}
	E'(\gamma(\cdot,t)) &=
	\sigma_1\Big(
		-\sigma_1\int_\gamma k^2\,ds - 2\pi\omega\sigma_2
	 	\Big)
	+ \sigma_2\Big(
		-2\pi\omega\sigma_1 - \sigma_2L(\gamma(\cdot,t))
	 	\Big)
\\&=
		- \int_\gamma \sigma_1^2k^2 + 2\sigma_1\sigma_2k + \sigma_2^2\,ds
   =
		- \int_\gamma F^2\,ds\,,
\end{align*}
as required.
\end{proof}

It follows immediately from the above lemmata and our assumptions that
$\sigma_1, \sigma_2 > 0$ that length and area are uniformly bounded along the
flow.
Lemma \ref{LMenergy} additionally yields finite maximal existence time in the following sense.

\begin{lem}[$T<\infty$]
	\prea
	Suppose that the winding number of the initial data is positive.
	Then $T<\infty$.
\label{Tfinite}
\end{lem}
\begin{proof}
We first note that (using Lemma \ref{curvatureevolution} below)
\[
	\frac{d}{dt}\int_\gamma k\,ds = \int_\gamma \sigma_1k_{ss} + \sigma_1k^3 + \sigma_2k^2 - k^2(\sigma_1k + \sigma_2)\,ds
	 = 0
\]
so that the winding number is constant along the flow, and in particular remains positive.

Suppose $T=\infty$.
Then
\[
	\int_0^{t_i} \vn{F}_2^2(t)\,dt =  E(\gamma(\cdot,0)) - E(\gamma(\cdot,t_i)) \le E(\gamma(\cdot,0)) < \infty \,,
\]
so there exists a subsequence $t_j\rightarrow\infty$ such that $\vn{F}_2^2(t_j)\searrow0$.
By smoothness and uniform boundedness of $L$, this implies that there exists a sequence $\varepsilon_j\searrow0$ such that
\[
	|\sigma_1k(u,t_j) + \sigma_2| \le \varepsilon_j\,.
\]
That is,
\[
	-\sigma_1^{-1}(\varepsilon_j+\sigma_2) \le k(u,t_j) \le \sigma_1^{-1}(\varepsilon_j - \sigma_2)\,.
\]
This means that for some sufficiently large $j$, $k(u,t_j)$ is uniformly
negative. But this is impossible, since $\int_\gamma k\,ds = 2\omega\pi > 0$.
\end{proof}

One key property of the flow is the preservation of local convexity.
For this, we first need the evolution of the curvature.

\begin{lem}[Curvature evolution]
\prea
	The evolution of curvature is given by
\begin{equation}\label{C-E k_t(s)}
k_t =\sigma_1 k_{ss}+ \sigma_1 k^3 + \sigma_2k^2.
\end{equation}
\label{curvatureevolution}
\end{lem}
\begin{proof}
We use Lemma~\ref{C-E commutator} and Lemma~\ref{tg nml evolution} to compute
\begin{align*}
k_t &= \IP{\gamma_{ss}}{\nu}_t = \IP{(\gamma_s)_{st}}{\nu} + \IP{\gamma_{ss}}{\nu_t}
\\
&= \IP{(\gamma_s)_{ts}+kF\gamma_{ss}}{\nu} + \IP{k\nu}{-F_s\tau}
\\
&= \IP{(F_s\nu)_s + k^2F\nu}{\nu}
\\
&= F_{ss}+k^2F = \sigma_1k_{ss}+\sigma_1k^3+\sigma_2k^2.
\end{align*}
\end{proof}

\begin{cor}[Convexity preservation]
\prea
Suppose $k(u,0) \ge k_0$.
Then $k(u,t) \ge k_0$.
\label{curvaturepreservation}
\end{cor}
\begin{proof}
	Apply the minimum principle to \eqref{C-E k_t(s)}.
\end{proof}


\section{Curvature estimates and finite time existence}
\label{sec:C-E estimate}

In this section we obtain estimates inspired by the classical work of Gage-Hamilton \cite{gage1986heat}.

Let us start with reparametrising the family of curves by the tangent angle
$\theta$, which is the angle between the tangent line and the $x$-axis. This is
a convenient choice of parameter for the study of closed convex curves. We
have the following relationship between the angle parameter $\theta$ and the
arc-length parameter $s$.

\begin{lem}\label{C-E angle derivatives}
Let $\gamma:\S^1\times[0,T)\to\R^2$ be a family of closed, embedded, convex
	plane curves parametrised by tangent angle $\theta$ under the flow
	\eqref{C-E problem}. We have 
$$\pD{\theta}{s}=k, \qquad \pD{\theta}{t}=F_s=\sigma_1k_s.$$
\end{lem}
\begin{proof}
For such a parametrisation, the unit tangent and unit normal of the curve has
the expression $\tau(\theta) = (\cos\theta, \sin\theta)$ and $\nu(\theta) =
(-\sin\theta, \cos\theta)$ respectively. We can therefore compute
\begin{align*}
\pD{\tau}{s} &= \pD{\tau}{\theta}\pD{\theta}{s} = (-\sin\theta,\cos\theta)\pD{\theta}{s} = \pD{\theta}{s}\,\nu;
\\
\pD{\tau}{t} &= \pD{\tau}{\theta}\pD{\theta}{t} = \pD{\theta}{t}\,\nu.
\end{align*}
The first equality of the lemma follows from comparing the top equation with
the Fernet-Serret equation $\pD{\tau}{s}=k\nu$. The second equality is obtained
by comparing the expression of $\pD{\tau}{t}$ to the corresponding one in
Lemma~\ref{tg nml evolution}.
\end{proof}

The space parameter $\theta$ does not commute with the time parameter $t$.
However, we can reparametrise in time by $t'$ such that $(\theta,t')$ are
independent. In the following, we reparametrise the
curves from $(u,t)$ to $(\theta,t')$ and define the new differential operator
$\partial_{t'}$ to be the time derivative taken along fixed $\theta$. We can
write down the evolution equation for $k$ with parameters $\theta$ and $t'$ as
the following.

\begin{lem}\label{C-E Lemma kPDE angle}
Let $\gamma:\S^1\times[0,T)\to\R^2$ be a family of closed, embedded, convex
	plane curves under the flow \eqref{C-E problem}. The evolution of
	curvature with respect to the parameters $(\theta,t')$ is given by
\begin{align}\label{C-E kPDE angle}
k_{t'} = k^2(F_{\theta\theta}+F) = \sigma_1 k^2 k_{\theta\theta} + \sigma_1 k^3 + \sigma_2k^2.
\end{align}
\end{lem}
\begin{proof}
This is a transformation of \eqref{C-E k_t(s)} via Lemma~\ref{C-E angle derivatives}. Note that
\begin{align*}
k_s = \frac{\partial k}{\partial \theta}\frac{\partial \theta}{\partial s}
= k_\theta k, \qquad
k_{ss} = k\frac{\partial}{\partial \theta}\left(k\frac{\partial k}{\partial \theta}\right) = kk_\theta^2 + k^2k_{\theta\theta},
\end{align*}
and 
\begin{align*}
k_t = \frac{\partial k}{\partial t'}\frac{\partial t'}{\partial t} + \frac{\partial k}{\partial \theta}\frac{\partial \theta}{\partial t}
= \frac{\partial k}{\partial t'} + \frac{\partial k}{\partial \theta}(\sigma_1k_s)
= k_{t'} + \sigma_1 k k_\theta^2.
\end{align*}
Substituting these into \eqref{C-E k_t(s)}, the result follows.
\end{proof}

For the rest of the paper, we abuse notation and replace $t'$ by $t$ for
simplicity whenever the tangent angle is used as the space parameter.

Next, we present a result found in~\cite{gage1986heat}, which can be viewed as
a version of the fundamental theorem of curves for simple closed convex plane
curves. We provide a proof for the convenience of the reader.

\begin{lem}[Gage-Hamilton \cite{gage1986heat}]
A positive $2\pi$ periodic function $k(\theta)$ represents the curvature function of a simple closed strictly convex $C^2$ plane curve $\gamma$ if and only if
\begin{equation}\label{fund thm of curve (eqn)}
\int_0^{2\pi}\frac{\cos\theta}{k(\theta)}\,d\theta = \int_0^{2\pi}\frac{\sin\theta}{k(\theta)}\,d\theta = 0.
\end{equation}
\label{fund thm of curves}
\end{lem}

\begin{proof}
Let $k:\S^1\to\R$ be the curvature function of a unit speed curve. As the curve is closed, we must have
\begin{align*}
0 = \int_0^L \tau \,ds = \int_0^L \tau \frac{1}{k(\theta)}\,d\theta = \int_0^{2\pi} \left(\frac{\cos\theta}{k(\theta)},\frac{\sin\theta}{k(\theta)}\right)\,d\theta.
\end{align*}
This proves one direction of the claim.

To see the other direction, suppose that $k:\S^1\to\R$ is a positive $2\pi$ periodic function satisfying~\eqref{fund thm of curve (eqn)}. We claim
\begin{equation}\label{reconstructed curve}
\gamma(\theta)=\left(\int_0^{\theta}\frac{\cos\theta'}{k(\theta')}\,d\theta', \int_0^{\theta}\frac{\sin\theta'}{k(\theta')}\,d\theta' \right)
\end{equation}
represents the associated curve in the plane up to translation and rotation, i.e., isometries of $\R^2$.

Let 
\begin{equation}\label{reconstructed position vectors}
x(\theta)= \int_0^{\theta}\frac{\cos\theta'}{k(\theta')}\,d\theta' \quad\text{ and }\quad y(\theta)=\int_0^{\theta}\frac{\sin\theta'}{k(\theta')}\,d\theta'.
\end{equation} 

As both $\cos\theta'$ and $k(\theta')$ are $2\pi$ periodic functions, we must have $\frac{\cos\theta'}{k(\theta')}$ and hence $x(\theta)$ also $2\pi$ periodic. Similarly, as $\sin\theta'$ is $2\pi$ periodic as well, we conclude that $y(\theta)$ must also be $2\pi$ periodic. Since the position vector is $2\pi$ periodic, the reconstructed curve $\eta(\theta)=(x(\theta),y(\theta))$ must be closed.

Let us compute the tangent vector $\vec{T}(\theta)$ for $\eta(\theta)=(x(\theta),y(\theta))$ using~\eqref{reconstructed position vectors},
\[
\vec{T}(\theta):=\eta_{\theta} = \left(\pD{}{\theta}\int_0^{\theta}\frac{\cos\theta'}{k(\theta')}\,d\theta' , \pD{}{\theta}\int_0^{\theta}\frac{\sin\theta'}{k(\theta')}\,d\theta' \right) = \left(\frac{\cos\theta}{k(\theta)}, \frac{\sin\theta}{k(\theta)} \right).
\]
Hence $|\eta_{\theta}|=\tfrac{1}{k(\theta)}$ and the unit tangent 
$\tau(\theta) = (\cos\theta,\sin\theta)$. We also have the unit normal $\nu(\theta)=(-\sin\theta,\cos\theta)$. The curvature scalar of $\eta(\theta)$ can be computed via
\[
k(\eta(\theta))= \IP{\frac{1}{|\eta_{\theta}|}\left(\frac{\eta_{\theta}}{|\eta_{\theta}|}\right)_{\theta}}{\nu} = \IP{k(\theta)(-\sin\theta,\cos\theta)}{(-\sin\theta,\cos\theta)} = k(\theta).
\]
Thus we conclude the function $k(\theta)$ represents the curvature function of $\eta(\theta)=(x(\theta),y(\theta))$.                                                                                                                                                                                                                                                                                                                                                                                                                                                                                                                                                                                                                                                                                                                                                                                                                                                                                                         

Hence $\gamma(\theta)$ as defined in~\eqref{reconstructed curve} represents the same curve as $\eta(\theta)$ and we have proved the claim.
\end{proof}

\begin{thm}
The flow problem~\eqref{C-E problem} is equivalent to the initial value PDE problem: 

Find $k:\S^1\times[0,T)\to \R$ satisfying
\begin{enumerate}
\item $k\in C^{2+\alpha, 1+\alpha}(\S^1\times[0,T-\varepsilon])$ for all $\varepsilon >0$.
\item $k_{t} = \sigma_1 k^2 k_{\theta\theta} + \sigma_1 k^3 + \sigma_2k^2$.\label{Eqn2 in thm}
\item $k(\theta,0)= \psi(\theta)$ where $\psi \in C^{1+\alpha}(\S^1)$ is strictly positive and satisfies $$\int_0^{2\pi}\frac{\cos\theta}{\psi(\theta)}\,d\theta = \int_0^{2\pi}\frac{\sin\theta}{\psi(\theta)}\,d\theta = 0.$$
\end{enumerate}
\label{C-E PDE problem}
\end{thm}

\begin{proof}
It is a direct consequence of Lemma~\ref{C-E Lemma kPDE angle} and
Lemma~\ref{fund thm of curves} that a solution to~\eqref{C-E problem} leads to
a solution to the above initial value PDE system. On the other hand, given a
solution to Theorem~\ref{C-E PDE problem} part \eqref{Eqn2 in thm}, we are able to
re-construct the family of curves satisfying~\eqref{C-E problem} up to
translation and rotation. This is achieved using the
formula~\eqref{reconstructed curve}. The initial condition of $k(\theta)$
ensures it can be viewed as a curvature function to an initial curve moving by the flow, and the initial curve is again constructed
via~\eqref{reconstructed curve}.
\end{proof}

The change of parametrisation does not affect the preservation of convexity
(Corollary \ref{curvaturepreservation}).
We may phrase this in the context of Theorem \ref{C-E PDE problem} as follows.

\begin{lem}
If $k:\S^1\times[0,T)\to \R$ satisfies the assumptions of Theorem~\ref{C-E PDE
	problem}, then $k_{\min} (t) = \inf\{k(\theta,t)|0\leq\theta \leq
	2\pi\}$ is a nondecreasing function.
\end{lem}

We next show that the curvature $k$ has a uniform bound if the area is uniformly bounded from below.

\begin{thm}[Curvature bounds]
Suppose $k:\S^1\times[0,T)\to \R$ satisfies the assumptions of Theorem~\ref{C-E
	PDE problem} and that the area enclosed by the associated curves is
	uniformly bounded away from zero.
	Then there exists constants $C_p$ depending only on $\sigma_1,\sigma_2,T$ and $\alpha(p)$ where
\[
\alpha(p) = \sum_{j=0}^p \max |\partial_\theta^jk(\theta,0)|\,,
\]
such that $\vn{k_{\theta^p}}_\infty \le C_p$.
\label{C-E thm curvature bound}
\end{thm}

We will present a proof that follows the curve shortening case as
in~\cite{gage1986heat} with three steps to complete: the geometric estimate,
the integral estimate and the pointwise estimate. The assumption of enclosed
area bounded away from zero implies the length must also be strictly positive.
These together provide a foundation for the geometric estimate and the rest of
the argument.

In order to develop the geometric and integral estimates, we require the
concept of the median curvature $k^*$:
\[
k^* = \sup\{b \,:\, k(\theta) > b\text{ on some interval of length } \pi\}\,.
\]

\begin{prop}[Geometric estimate \cite{gage1986heat}]
Let $\gamma:\S^1\times[0,T)\to\R^2$ be a family of convex closed plane curve
	with curvature function $k:\S^1\times[0,T)\to\R$, corresponding
		enclosed area $A:[0,T)\to\R$ and length $L:[0,T)\to\R$. The
				relation $k^*(t) < L/A$ holds.
\label{GeomEst}
\end{prop}
\begin{proof}
This result is independent of the flow. It was first proved in
\cite{gage1986heat}. Here we give for the convenience of the reader an overview
of the proof.

Given that the median curvature satisfies $k^*(t)>M$, the curve $\gamma$
restricted to an interval $(a, a+\pi)$ has curvature $k(\theta,t)>M$. This
segment of the curve can be contained in a circle of radius $1/M$, and hence
between two parallel lines whose distance is at least $2/M$ apart. This also
implies the entire curve lies between these two parallel lines as the curve is
convex. We further deduce that the convex curve can be contained in a
rectangular box with width $2/M$ and length $L/2$. Comparing the enclosed area
of $\gamma$ and the area of the box yields
\[
	A(\gamma)< A(\text{rectangular box}) = \frac{2}{M}\frac{L}{2}=\frac{L}{M}\,.
\]
Allowing $M$ to become arbitrarily close to $k^*$, we have $k^*(t)< L/A$ as
required.
\end{proof}

The second step is the integral estimate under a bound on the median curvature.

\begin{prop}[Integral estimate]
Suppose $k:\S^1\times[0,T)\to \R$ satisfies the assumptions of Theorem~\ref{C-E
	PDE problem} with $k(\theta,0) > k_0 > 0$.
Suppose for each $t\in [0,T)$, $k^*(t)<M<\infty$. Then we have
		the following uniform bound for the entropy along the flow:
		\[
\int_0^{2\pi} \log k(\theta,t) \,d\theta 
\leq \int_0^{2\pi} \log k\,d\theta\bigg|_{t=0} + \left(2M+\frac{\sigma_2}{\sigma_1}\right)L(0)+ 2\pi\sigma_1M^2T\,.
		\]
\label{IntegEst}
\end{prop}
\begin{proof}
Using the evolution of curvature and integration by parts we obtain
\begin{align*}
\frac{d}{dt}\int_0^{2\pi} \log k \,d\theta 
&= \int_0^{2\pi} \frac{1}{k}\frac{\partial k}{\partial t} \,d\theta
= \int_0^{2\pi} \sigma_1 k k_{\theta\theta} + \sigma_1 k^2 + \sigma_2k \,d\theta \\
&= \sigma_1 \int_0^{2\pi} k^2 -k_{\theta}^2\,d\theta + \sigma_2\int_0^{2\pi} k\,d\theta.
\end{align*}
To estimate the first integral above, we see that for a fixed $t$, the space
domain is comprised of two distinct subsets: the open set
$U=\{\theta|k(\theta,t)>k^*(t)\}$ and its complement set $V=\S^1-U$. The
definition of median curvature implies that the open set $U$ is a countable union of
disjoint intervals $I_i$, each of length no bigger than $\pi$. We can apply the
Wirtinger's inequality to the function $k(\theta,t)-k^*(t)$
in the closure of each interval $I_i$ to see
\[
\int_{\bar{I_i}}(k-k^*)^2\,d\theta \leq \int_{\bar{I_i}}k_\theta^2\,d\theta\,.
\]
Rearrange and noting that $k^*$ is positive by convexity, we have
$$\int_{\bar{I_i}} k^2-k_\theta^2\,d\theta \leq 2k^*\int_{\bar{I_i}}k\,d\theta - 2\int_{\bar{I_i}}(k^*)^2\,d\theta \leq 2k^*\int_{\bar{I_i}}k\,d\theta.$$
Taking the union of the sets $I_i$, we obtain
$$\sigma_1\int_{U} k^2-k_\theta^2\,d\theta \leq 2\sigma_1k^*\int_{U}k\,d\theta \leq 2\sigma_1k^*\int_0^{2\pi} k\,d\theta.$$
On the complement set, we have $k(\theta,t)\leq k^*(t)$, hence
$$\sigma_1\int_{V} k^2-k_\theta^2\,d\theta \leq \sigma_1\int_{V} k^2\,d\theta \leq \sigma_1\int_0^{2\pi} k^2\,d\theta \leq 2\pi\sigma_1 (k^*)^2.$$
We therefore get a bound on the whole interval
$$\sigma_1 \int_0^{2\pi} k^2 -k_{\theta}^2\,d\theta \leq 2\sigma_1 k^*\int_0^{2\pi} k\,d\theta + 2\pi\sigma_1 \left(k^*\right)^2.$$
Recalling the evolution of length from Proposition~\ref{lenghtevolution} and $d\theta = k \,ds$, we obtain
\begin{align*}
\frac{d}{d t}\int_0^{2\pi} \log k \,d\theta 
&\leq 2\sigma_1 k^*\int_0^{2\pi} k\,d\theta + 2\pi\sigma_1 \left(k^*\right)^2 + \sigma_2\int_0^{2\pi} k\,d\theta
\\
&\leq \left(2k^* +\frac{\sigma_2}{\sigma_1 }\right)\left(-L_t-2\pi\sigma_2\right)+2\pi\sigma_1(k^*)^2.
\end{align*}
Note that here we used that $\omega = 1$, which follows immediately from convexity and embeddedness of the flow.
Assume that $k^*<M$ and integrate to obtain
\begin{align*}
\int_0^{2\pi} \log k(\theta,t) \,d\theta 
&\leq \int_0^{2\pi} \log k \,d\theta\bigg|_{t=0} + \left(2M+\frac{\sigma_2}{\sigma_1}\right)(L(0)-L(t)-2\pi\sigma_2t)+ 2\pi\sigma_1M^2t \\
&\leq \int_0^{2\pi} \log k \,d\theta\bigg|_{t=0} + \left(2M+\frac{\sigma_2}{\sigma_1}\right)L(0)+ 2\pi\sigma_1M^2t,
\end{align*}
for all $t<T$.
\end{proof}
Next, we shall upgrade the integral estimate to a pointwise estimate.
\begin{lem}
Suppose $k:\S^1\times[0,T)\to \R$ satisfies the assumptions of Theorem~\ref{C-E
	PDE problem} with $k(\theta,0) > k_0 > 0$.
The following estimate holds:
\[
\int_0^{2\pi}k_{\theta}^2\,d\theta
\leq 
   \int_0^{2\pi} k_{\theta}^2-k^2\,d\theta\bigg|_{t=0}
 + \int_0^{2\pi}k^2 \,d\theta + \frac{\sigma_2^2}{2\sigma_1}\int_0^t \int_0^{2\pi}k^2 \,d\theta dt
\,.
\]
\end{lem}
\begin{proof}
We compute
\begin{align*}
\frac{d}{d t}\int_0^{2\pi} k^2 -k_{\theta}^2\,d\theta 
&= 2\int_0^{2\pi}(kk_t-k_\theta k_{\theta t})\,d\theta
= 2\int_0^{2\pi}(k_{\theta\theta}+k) k_{t}\,d\theta \\
&=  2\int_0^{2\pi}(k_{\theta\theta}+k)\left(\sigma_1 k^2 k_{\theta\theta} + \sigma_1 k^3 + \sigma_2k^2 \right) \,d\theta \\
&= 2\sigma_1\int_0^{2\pi}(k_{\theta\theta} +k)^2k^2\,d\theta + 2\sigma_2\int_0^{2\pi}(k_{\theta\theta} +k)k^2\,d\theta.
\end{align*}
We apply Cauchy's inequality to obtain a lower bound for the last term with $\varepsilon>0$ to be chosen:
\begin{align*}
2\sigma_2\int_0^{2\pi}(k_{\theta\theta} +k)k^2\,d\theta
&\geq -2\sigma_2\int_0^{2\pi}|(k_{\theta\theta} +k)k||k|\,d\theta\\
&\geq -2\sigma_2\varepsilon\int_0^{2\pi}(k_{\theta\theta} +k)^2k^2\,d\theta - \frac{2\sigma_2}{4\varepsilon}\int_0^{2\pi} k^2\,d\theta.
\end{align*}
Choosing $\varepsilon = \sigma_1/\sigma_2$, we obtain
\begin{align*}
2\sigma_2\int_0^{2\pi}(k_{\theta\theta} +k)k^2\,d\theta
\geq -2\sigma_1\int_0^{2\pi}(k_{\theta\theta} +k)^2k^2\,d\theta - \frac{\sigma_2^2}{2\sigma_1}\int_0^{2\pi} k^2\,d\theta.
\end{align*}
Hence
\begin{align*}
\frac{d}{d t}\int_0^{2\pi} k_{\theta}^2-k^2\,d\theta \le \frac{\sigma_2^2}{2\sigma_1}\int_0^{2\pi} k^2\,d\theta.
\end{align*}
Integrating both sides in time, the conclusion follows.
\end{proof}

\begin{prop}[Pointwise estimate]
Suppose $k:\S^1\times[0,T)\to \R$ satisfies the assumptions of Theorem~\ref{C-E
	PDE problem} with $k(\theta,0) > k_0 > 0$.
Suppose
	\[
\int_0^{2\pi}\log k(\theta , t) \,d\theta \le C_1\,.
	\]
Then
\[
k_{\max}(t)
 \le 
 2e^{4(C_1 + 2\pi|\log k_0 |)\left(\sqrt{2\pi} + \sigma_2\sqrt{\frac{\pi T}{\sigma_1}}\right)^2}
 +
 \left(\sqrt{2\pi} + \sigma_2\sqrt{\frac{\pi T}{\sigma_1}}\right)^{-1}
 \sqrt{
 \int_0^{2\pi} k_{\theta}^2-k^2\,d\theta\bigg|_{t=0}}
 \,.
\]
In particular $k(\theta , t)$ is uniformly bounded on $\S^1\times[0,T)$.
\label{PointwEst}
\end{prop}
\begin{proof}
Preservation of convexity implies that $k(\theta,t)>k_0$ for all $\theta$ and
$t$. Fix a time $t_0$ and a positive number $V$, and consider the set
$I_V=\{\theta: \log k(\theta ,t_0)\geq V\}$ containing all points such that
$\log k(\theta ,t_0)$ is not less than $V$. We have
\begin{align*}
C_1\geq \int_0^{2\pi}\log k(\theta,t_0) \,d\theta &= \int_{I_V}\log k(\theta,t_0) \,d\theta + \int_{\S^1\backslash I_V}\log k(\theta,t_0) \,d\theta
\\
&\geq V\mu_L(I_V) + \log(k_0) \mu_L(\S^1\backslash I_V)
\end{align*}
where $\mu_L(\cdot)$ denotes the Lebesgue measure of a set. Rearrange to see that
\begin{align*}
V\mu_L(I_V) &\leq C_1-\log(k_0) \mu_L(\S^1\backslash I_V)
\\
&\leq C_1 + |\log(k_0)| \mu_L(\S^1\backslash I_V)
\\
&\leq C_1 + |\log(k_0)| \mu_L(\S^1).
\end{align*}
Let $C_2=C_1 + |\log(k_0)| \mu_L(\S^1)$, then $\mu_L(I_V)\leq \frac{C_2}{V}$.
Fixing $\delta=\frac{C_2}{V}$, we have $k(\theta,t_0)\leq e^{\frac{C_2}{\delta}}$
for all $\theta\notin I_V$.
We have for any $\varphi\in\S^1$, $a\notin I_V$,
\begin{align*}
k(\varphi)
&=k(a)+\int_a^{\varphi} k_\theta\,d\theta 
\\
&\leq e^{\frac{C_2}{\delta}} + \sqrt{\delta}\left( \int_0^{2\pi}k_{\theta}^2\,d\theta \right)^\frac{1}{2} \\
&\leq e^\frac{C_2}{\delta} + \sqrt{\delta}\left( \int_0^{2\pi}k^2 \,d\theta + \frac{\sigma_2^2}{2\sigma_1}\int_0^t \int_0^{2\pi}k^2 \,d\theta dt +
   \int_0^{2\pi} k_{\theta}^2-k^2\,d\theta\bigg|_{t=0}
\right)^\frac{1}{2},
\end{align*}
by the previous Lemma. Thus, suppose $k_{\max}$ is the maximum value of $k$, then
\begin{align*}
k_{\max} 
&\leq e^\frac{C_2}{\delta} + \sqrt{\delta}\left( 2\pi k_{\max}^2 + \frac{\sigma_2^2}{2\sigma_1}\int_0^t 2\pi k_{\max}^2 \,dt +\int_0^{2\pi} k_{\theta}^2-k^2\,d\theta\bigg|_{t=0} \right)^\frac{1}{2} \\
&\leq e^\frac{C_2}{\delta} + \sqrt{2\pi\delta}k_{\max} + \sqrt{\frac{\pi T\delta}{\sigma_1}}\sigma_2k_{\max} +\sqrt{\delta
   \int_0^{2\pi} k_{\theta}^2-k^2\,d\theta\bigg|_{t=0}
}\,.
\end{align*}
Absorbing yields
\begin{align*}
k_{\max} 
&\leq \frac{e^\frac{C_2}{\delta}+\sqrt{\delta
   \int_0^{2\pi} k_{\theta}^2-k^2\,d\theta\bigg|_{t=0}
}}{1-\sqrt{2\pi\delta}-\sqrt{\frac{\pi T\delta}{\sigma_1}}\sigma_2}.
\end{align*}
Choose finally
$\delta = \frac14\bigg(\sqrt{2\pi} + \sigma_2\sqrt{\frac{\pi T}{\sigma_1}}\bigg)^{-2}$,
so that the above becomes
\[
k_{\max} 
 \le 
 2e^{4C_2\left(\sqrt{2\pi} + \sigma_2\sqrt{\frac{\pi T}{\sigma_1}}\right)^2}
 +
 \left(\sqrt{2\pi} + \sigma_2\sqrt{\frac{\pi T}{\sigma_1}}\right)^{-1}
 \sqrt{
 \int_0^{2\pi} k_{\theta}^2-k^2\,d\theta\bigg|_{t=0}}
\]
as required.
\end{proof}
In the rest of this section, we show that assuming $k$ is bounded, we can find bounds for all higher derivatives of $k$.
Since we obtain the curvature bound so long as the area is positive, this implies that the flow continues to smoothly exist until it shrinks to a point.
We begin with a series of lemmata.
\begin{lem}
Suppose $k:\S^1\times[0,T)\to \R$ satisfies the assumptions of Theorem~\ref{C-E
	PDE problem} with $k(\theta,0) > k_0 > 0$.
If  $k(\theta,t)< k_{\max} <\infty$, then
\[
|k_\theta(\theta,t)| \le e^{2T(3\sigma_1 k_{\max}^2 + 2\sigma_2k_{\max})}\max |k_\theta(\theta,0)|\,.
\]
\label{C-E k'bound}
\end{lem}
\begin{proof}
We use the maximum priciple to prove that $k_\theta$ grows at most exponentially, that is $$ k_\theta^2(\theta,t)\leq e^{-2\alpha t}k_\theta^2(\theta,0),$$
for some negative constant $\alpha$ on a finite time interval $[0,T)$. 

We first compute
$$k_{\theta t} = \sigma_1\left(k^2k_{\theta\theta\theta} + 2kk_\theta k_{\theta\theta} + 3k^2k_\theta \right) + 2\sigma_2kk_\theta.$$
Let $X=e^{\alpha t} k_\theta$, then the above equation can be rewritten as
\begin{align*}
e^{-\alpha t}(X_t -\alpha X) &= e^{-\alpha t} \left[\sigma_1\left(k^2X_{\theta\theta} + 2e^{-\alpha t}kXX_\theta + 3k^2X \right) + 2\sigma_2kX  \right], \text{ so} \\
X_t-\sigma_1k^2X_{\theta\theta} &= 2e^{-\alpha t}\sigma_1kXX_{\theta} + \left(3\sigma_1 k^2 +2\sigma_2k +\alpha\right)X.
\end{align*} 
Therefore,
\begin{align*}
\left(\partial_t-\sigma_1 k^2\partial_{\theta\theta}\right)X^2
&= 2XX_t-\sigma_1k^2(2X_\theta^2+2XX_{\theta\theta})
\\
&= 2X\left(X_t-\sigma_1 k^2X_{\theta\theta}
\right) - 2\sigma_1k^2 X_\theta^2\\
&= 2\left(3\sigma_1 k^2 +2\sigma_2k +\alpha\right)X^2 + 2e^{-\alpha t}\sigma_1kX(2XX_\theta) -2\sigma_1k^2X_\theta^2 \\
&\leq 2\left(3\sigma_1 k^2 +2\sigma_2k +\alpha\right)X^2 + 2e^{-\alpha t}\sigma_1kX(X^2)_\theta .
\end{align*}
Suppose $k$ is bounded uniformly by $k_{\max}$.
Then we can choose $\alpha < -3\sigma_1 k_{\max}^2 - 2\sigma_2k_{\max}$ so that
the coefficient of $X^2$ is negative.
Then, the conclusion follows by the maximum pinciple.
\end{proof}
\begin{lem}
Suppose $k:\S^1\times[0,T)\to \R$ satisfies the assumptions of Theorem~\ref{C-E
	PDE problem} with $k(\theta,0) > k_0 > 0$.
If $k<k_{\max}<\infty$ and $|k_\theta|\le C_3$, then
\[
\int_0^{2\pi} k_{\theta\theta}^4\,d\theta
\le \bigg(2T\pi + 
\int_0^{2\pi} k_{\theta\theta}^4\,d\theta\bigg|_{t=0}\bigg)e^{C_3^2(36\sigma_1 + C_3^2(81\sigma_1k_{\max}^2 + \sigma_2^2\sigma_1^{-1}36)^2)T}
\,.
\]
\end{lem}
\begin{proof}
We compute the following using the evolution equation and integration by parts
\begin{align*}
\frac{d}{d t}	\int_0^{2\pi}k_{\theta\theta}^4 \,d\theta
&= 4\int_0^{2\pi} k_{\theta\theta}^3 \frac{\partial}{\partial t}(k_{\theta\theta}) \,d\theta
= 4\int_0^{2\pi} k_{\theta\theta}^3 \frac{\partial^2}{\partial \theta^2}(k_{t}) \,d\theta \\
&= 4\int_0^{2\pi} k_{\theta\theta}^3 \left( \sigma_1 k^2 k_{\theta\theta} + \sigma_1 k^3 + \sigma_2k^2 \right)_{\theta\theta}\,d\theta \\
&= -12\int_0^{2\pi} k_{\theta\theta}^2 k_{\theta\theta\theta} \left( \sigma_1 k^2 k_{\theta\theta} + \sigma_1 k^3 + \sigma_2k^2 \right)_{\theta}\,d\theta \\
&= -12\int_0^{2\pi} k_{\theta\theta}^2 k_{\theta\theta\theta} \left(\sigma_1k^2k_{\theta\theta\theta} + 2\sigma_1kk_\theta k_{\theta\theta} +3\sigma_1k^2k_\theta + 2\sigma_2kk_\theta\right)\,d\theta
\\
&= -12\int_0^{2\pi} \sigma_1k^2k_{\theta\theta}^2k_{\theta\theta\theta}^2 + 2\sigma_1kk_\theta k_{\theta\theta}^3k_{\theta\theta\theta} + 3\sigma_1k^2k_\theta k_{\theta\theta}^2k_{\theta\theta\theta} + 2\sigma_2kk_\theta k_{\theta\theta}^2k_{\theta\theta\theta}\,d\theta.
\end{align*}
We use Cauchy's inequality to absorb the last three terms into the first term. We obtain
\begin{align*}
\frac{d}{d t}	\int_0^{2\pi}k_{\theta\theta}^4 \,d\theta
\leq \int_0^{2\pi} c_1 k_\theta^2 k_{\theta\theta}^4 + c_2 k^2k_\theta^2 k_{\theta\theta}^2 + c_3k_\theta^2k_{\theta\theta}^2\,d\theta,
\end{align*}
for some constants $c_1$, $c_2$ and $c_3$ depending on $\sigma_1$ and $\sigma_2$ only. Moreover, the H\"older inequality implies
$$\int_0^{2\pi} k_{\theta\theta}^2 \,d\theta \leq \sqrt{2\pi \left(\int_0^{2\pi} k_{\theta\theta}^4 \,d\theta \right)}.$$
Combining this with the assumption that $k<k_{\max}$ and $|k_\theta| \le C_3$ is bounded, we deduce that on a finite time interval
\begin{align*}
\frac{d}{d t}	\int_0^{2\pi}k_{\theta\theta}^4 \,d\theta
&\leq \int_0^{2\pi} c_1C_3^2 k_{\theta\theta}^4 + \left(c_2 k_{\max}^2C_3^2 + c_3C_3^2\right)k_{\theta\theta}^2 \,d\theta
\\
&\leq C_3^2(c_1 + C_3^2(c_2k_{\max}^2 + c_3)^2)\int_0^{2\pi}k_{\theta\theta}^4 \,d\theta + 2\pi
\,.
\end{align*}
Using Gronwall's inequality we see that $\int_0^{2\pi}k_{\theta\theta}^4\,d\theta$ grows at most exponentially.
Returning to the choice of $c_1, c_2, c_3$ we observe that allowable choices are
\[
c_1 = 36\sigma_1\,,\qquad
c_2 = 81\sigma_1\,,\qquad
c_3 = 36\frac{\sigma_2^2}{\sigma_1}\,.
\]
\end{proof}

\begin{lem}
Suppose $k:\S^1\times[0,T)\to \R$ satisfies the assumptions of Theorem~\ref{C-E
	PDE problem} with $k(\theta,0) > k_0 > 0$.
If $k<k_{\max}<\infty$, $|k_\theta|\le C_3$, and $\vn{k_{\theta\theta}}_4^4 \le C_4$ then
\[
\int_0^{2\pi} k_{\theta\theta\theta}^2\,d\theta
 < 
    \bigg(
	D_2t + \int_0^{2\pi} k_{\theta\theta\theta}^2\,d\theta\bigg|_{t=0}
    \bigg)
    e^{D_1t}
\,,
\]
where $D_1 = 56\sigma_1C_3^2$ and
\[
D_2 = 28\sigma_1\pi\bigg[2\bigg(C_4 + \frac18k_0^{-4}C_3^8\bigg)
                + \frac{9}{4}\bigg(C_4 + \frac14k_{\max}^4\bigg) + 9C_3^4 
		+ \sigma_2^2\sigma_1^{-2}\bigg(k_0^{-2}C_3^4 + C_4 + \frac14\bigg)\bigg]
\,.
\]
\label{k''bound}
\end{lem}
\begin{proof}
In this proof we use $k^{(4)}$ to denote the fourth order derivative of $k$. We apply integration by parts to find
\begin{align*}
\frac{d}{d t}	\int_0^{2\pi}k_{\theta\theta\theta}^2\,d\theta
&= 2\int_0^{2\pi}k_{\theta\theta\theta} (k_t)_{\theta\theta\theta}\,d\theta
= -2\int_0^{2\pi}k^{(4)} (k_t)_{\theta\theta}\,d\theta \\
&= -2\int_0^{2\pi}k^{(4)} \left( \sigma_1 k^2 k_{\theta\theta} + \sigma_1 k^3 + \sigma_2k^2 \right)_{\theta\theta}\,d\theta \\
&= -2\int_0^{2\pi}k^{(4)} \left( \sigma_1 k^2k_{\theta\theta\theta} + 2\sigma_1 kk_\theta k_{\theta\theta}+ 3\sigma_1k^2k_\theta + 2\sigma_2kk_\theta \right)_{\theta}\,d\theta \\
&= -2\sigma_1\int_0^{2\pi} k^2 \left(k^{(4)}\right)^2 + 4kk_\theta k_{\theta\theta\theta}k^{(4)} + 2kk_{\theta\theta}^2k^{(4)} + 2k_\theta^2k_{\theta\theta}k^{(4)} \\
&\qquad + 3k^2k_{\theta\theta}k^{(4)} + 6kk_\theta^2 k^{(4)} \,d\theta
-4\sigma_2\int_0^{2\pi} k_\theta^2k^{(4)} + kk_{\theta\theta}k^{(4)}\,d\theta.
\end{align*}
Again, by applying Cauchy's inequality, we can absorb all the other terms into the first term with some
additional penalty terms. That is
\begin{align*}
\frac{d}{d t}	\int_0^{2\pi}k_{\theta\theta\theta}^2\,d\theta
&\leq c_1\int_0^{2\pi}k_\theta^2k_{\theta\theta\theta}^2\,d\theta + c_2\int_0^{2\pi}k_{\theta\theta}^4\,d\theta +c_3 \int_0^{2\pi} \frac{k_\theta^4}{k^2} k_{\theta\theta}^2\,d\theta +c_4\int_0^{2\pi} k^2 k_{\theta\theta}^2\,d\theta \\
&\quad + c_5\int_0^{2\pi} k_\theta^4\,d\theta + c_6\int_0^{2\pi} \frac{k_\theta^4}{k^2}\,d\theta + c_7\int_0^{2\pi}k_{\theta\theta}^2\,d\theta.
\end{align*}
Estimating $k_{\theta\theta}^2 \le k_{\theta\theta}^4 + \frac14$
and invoking our hypotheses we find
\begin{align*}
\frac{d}{d t} \int_0^{2\pi}k_{\theta\theta\theta}^2\,d\theta \leq D_1\int_0^{2\pi}k_{\theta\theta\theta}^2\,d\theta + D_2\,,
\end{align*}
for universal $D_1$ and $D_2$.
Applying Gronwall's inequality, we recover (i).
Following through on allowable constants we see that we may choose
\[
c_1 = 56\sigma_1\,,\quad
c_2 = c_3 = 14\sigma_1\,,\quad
c_4 = \frac{63}{2}\sigma_1\,,\quad
c_5 = 126\sigma_1\,,\quad
c_6 = c_7 = 14\frac{\sigma_2^2}{\sigma_1}\,.
\]
Then $D_1 = c_1C_3^2$ and 
\[
D_2 = 2\pi\bigg[c_2C_4 + c_3\bigg(C_4 + \frac14k_0^{-4}C_3^8\bigg) + c_4\bigg(C_4 + \frac14k_{\max}^4\bigg) + c_5C_3^4 + c_6k_0^{-2}C_3^4 + c_7\bigg(C_4 + \frac14\bigg)\bigg]
\,.
\]
\end{proof}

\begin{lem}
Suppose $k:\S^1\times[0,T]\to \R$ satisfies the assumptions of Theorem~\ref{C-E
	PDE problem} with $k(\theta,0) > k_0 > 0$.
If $k<k_{\max}<\infty$, then for all $p\in\N$ we have
\[
|k_{\theta^p}| \le C(\sigma_1,\sigma_2,T,k_{\max},\alpha(p))\,.
\]
\label{AllderivEst}
\end{lem}
\begin{proof}
We begin by collecting consequences of the assumed uniform curvature bound from above.
First, Lemma \ref{C-E k'bound} implies $|k_\theta| \le C(\sigma_1,\sigma_2,T,k_{\max},\alpha(1))$.
Then Lemma \ref{k''bound} gives, by the fundamental theorem of calculus, $|k_{\theta\theta}| \le C(\sigma_1,\sigma_2,T,k_{\max},\alpha(3))$.
To begin our general argument, we need to bound $k_{\theta\theta\theta}$.
We do this by the maximum principle.

Let us compute
\begin{align*} 
\frac{\partial}{\partial t}k_{\theta\theta\theta}
&= (k_t)_{\theta\theta\theta} = \left( \sigma_1 k^2 k_{\theta\theta} + \sigma_1 k^3 + \sigma_2k^2 \right)_{\theta\theta\theta} \\
&= \sigma_1\left[ k^2k^{(4)} + 4kk_\theta k_{\theta\theta\theta} + 2kk_{\theta\theta}^2 + 2k_\theta^2k_{\theta\theta} + 3k^2k_{\theta\theta} + 6kk_\theta^2 + \frac{2\sigma_2}{\sigma_1}\left(kk_{\theta\theta} + k_\theta^2\right)
 \right]_{\theta} \\
&= \sigma_1\left[ k^2k^{(5)} + 6kk_\theta k^{(4)} + \left(8kk_{\theta\theta}+6k_\theta^2+3k^2 + \frac{2\sigma_2}{\sigma_1}k\right) k_{\theta\theta\theta} \right] \\ 
&\quad + \sigma_1 \left(6k_\theta k_{\theta\theta}^2 +18kk_\theta k_{\theta\theta} +6k_\theta^3\right) + 6\sigma_2k_\theta k_{\theta\theta}.
\end{align*}
Since $k$, $k_\theta$ and $k_{\theta\theta}$ are all bounded on the finite time
interval $[0,T)$, the second line of the last inequality above can be bounded
by some constant $D = D(\sigma_1,\sigma_2,T,k_{\max},\alpha(3))$,
and the term in parentheses (the coefficient of $k_{\theta\theta\theta}$) may be bounded by a constant $E = E(\sigma_1,\sigma_2,T,k_{\max},\alpha(3))$.
Letting $Y = e^{\alpha t}k_{\theta\theta\theta}$, we can rewrite the above in
the same way as we did in Lemma~\ref{C-E k'bound}:
\begin{align*}
e^{-\alpha t}(Y_t-\alpha Y) \leq
     e^{-\alpha t}\sigma_1\left[k^2Y_{\theta\theta} + 6kk_\theta Y_\theta + EY\right] + D.
\end{align*}
Hence 
\begin{align*}
(\partial_t-\sigma_1k^2\partial_{\theta\theta})Y^2
&= 2Y(Y_t-\sigma_1k^2Y_{\theta\theta})-2\sigma_1k^2Y_\theta^2 \\
&\leq
     2\alpha Y^2 + 
 2\sigma_1EY^2 + 6\sigma_1(k_{\max})C(Y^2)_\theta +2e^{\alpha t}DY.
\end{align*}     
Suppose there exists a new maximum for $Y^2$ at the point $(\theta_0,t_0)$.
Note that we may assume $Y^2(\theta_0,t_0) > 1$ and $\alpha < 0$ so that
\begin{align*}
(\partial_t-\sigma_1k^2\partial_{\theta\theta})Y^2
&\leq
     Y^2(2\alpha + 2\sigma_1E + 2D)
    + 6\sigma_1(k_{\max})C(Y^2)_\theta 
\,.
\end{align*}
Picking $\alpha$ so negative that
\[
2\alpha + 2\sigma_1E + 2D < 0
\]
and then noting that at $(\theta_0,t_0)$ we have $(Y^2)_\theta = 0$, we have
$(\partial_t-\sigma_1k^2\partial_{\theta\theta})Y^2 < 0$, a contradiction.
Hence, $k_{\theta\theta\theta}^2(\theta,t)\le e^{-2\alpha t}k_{\theta\theta\theta}(\theta,0) \le e^{-2\alpha
T}\vn{k_{\theta\theta\theta}}_\infty^2$ and so $|k_{\theta\theta\theta}| \le C(\sigma_1,\sigma_2,T,k_{\max},\alpha(3))$.

We now show that this implies all the higher derivatives of $k$ are also bounded.
The above bounds imply that if $|k_{\theta^q}| \le C(\sigma_1,\sigma_2,T,k_{\max},\alpha(q))$
for all $q\in\{0,\ldots,p-1\}$, the following evolution equation for $k_{\theta^p}$ holds:
\[
(\partial_t-\sigma_1k^2\partial_{\theta\theta})k_{\theta^p}
 \le 2p\sigma_1 kk_\theta (k_{\theta^p})_\theta + Ck_{\theta^p} + C\,.
\] 
Using the substitution above and the assumed bounds, we see that $|k_{\theta^p}| \le C(\sigma_1,\sigma_2,T,k_{\max},\alpha(p))$.
The claim follows by induction.
\end{proof}

These estimates allow us to obtain qualitative information on the evolution of length and area up to final time.

\begin{thm}
\label{LAasympt}
If $\gamma(\cdot,0)$ is convex, 
then $A(\gamma(\cdot,t)) \rightarrow 0$ and 
$L(\gamma(\cdot,t)) \rightarrow 0$ as $t\nearrow T$.
There exists a final point $\FP\in\R^2$ such that
\[
\gamma(\S,t)\rightarrow\FP\quad\text{as}\quad t\rightarrow T\,.
\]
Here convergence is understood with respect to the Hausdorff metric on $\R^2$.
\end{thm}
\begin{proof}
Suppose that the area satisfies $A(\gamma(\cdot,t)) \rightarrow \varepsilon > 0$ as $t\nearrow T$.
Note that Lemma \ref{Tfinite} implies $T<\infty$.
Then Proposition \ref{GeomEst} (note that here and throughout this proof we use
convexity) yields a uniform estimate on the median curvature, which combined
with Proposition \ref{IntegEst} gives a uniform estimate on the entropy.
If $\gamma(\cdot,0)$ is convex, then it remains so (see Corollary \ref{curvaturepreservation}).
Note that we know by Lemma \ref{Tfinite} that $T<\infty$.  The uniform estimate
on
the entropy then by Proposition \ref{PointwEst} becomes a uniform estimate on
curvature depending only on $\sigma_1, \sigma_2, T$ and the initial values of
$k$ and $k_\theta$.
Finally by Lemma \ref{AllderivEst} all derivatives of curvature are uniformly
bounded.
Recalling the formulation of Theorem \ref{C-E PDE problem}, this implies the
extension of the solution beyond $T$ by a standard argument.
Since $T$ is finite, this is a contradiction.

Therefore the area satisfies
\[
A(\gamma(\cdot,t)) \rightarrow 0 \quad\text{ as }\quad t\nearrow T\,.
\]
This is the first conclusion of the theorem.
The only possible limiting shapes for $\gamma(\cdot,T)$ are straight line segments and points.
As the flow is smooth and closed for all $t\in[0,T)$, any other possibility is not convex up to final time.
If the limiting shape is a point, then we are finished.

If the limiting shape is a straight line segment, then
\[
	\min_{\theta\in\S} k(\theta,t) \rightarrow 0\quad\text{ as }\quad t\nearrow T
\]
But the maximum principle implies
\[
\min_{\theta\in\S} k(\theta,t) \ge 
\min_{\theta\in\S} k(\theta,0) > 0\,. 
\]
This is a contradiction.

Therefore the only possible limiting shape is a point, and so
$L(\gamma(\cdot,t)) \rightarrow 0$ as required.
\end{proof}


\section{Continuous rescaling}
\label{CRsection}

We will now study a rescaling of the flow about the final point $\FP$.
Our goal for the remainder of the section is to prove that in a weak sense a
subsequence converges to a circle.
There are essentially two steps to be completed.
Firstly, we further develop our a-priori estimates on the speed, which 
then implies the entropy of the rescaled flow is bounded.
Secondly, the entropy bound leads to a two way bound on the curvature and
implies subconvergence of the family of curves.

Consider the scaling factor $\phi:[0,T)\to\R$ given by
\begin{equation}\label{C-E Zhu scaling factor}
\phi(t) = (2T-2t)^{-1/2}\,
\end{equation}
and set the corresponding rescaled flow $\hat{\gamma}:\S^1\times[0,\infty)\to \R^2$ to be
\[
	\hat{\gamma}(\hat{\theta} , \hat{t}\,)
	= \phi(t(\hat t))\big[\gamma(\theta(\hat\theta) ,t(\hat t))-\FP\big]
        = \frac1{\sqrt{2T}} e^{\hat t}\big[\gamma(\theta(\hat\theta) ,t(\hat t))-\FP\big]\,,
\]
where $\FP$ is the final point given by Theorem \ref{LAasympt}, and
\[
	\theta(\hat \theta) = \hat\theta\,,\quad
	t(\hat t) = T(1 - e^{-2{\hat{t}}{ }})\,,\quad
	\hat{t}(t) = -\frac{1}{2}\log\left(1-\frac{t}{T}\right)
\]
denote the rescaled time and space variables.
Under the rescaling, we see that the flow \eqref{C-E problem} becomes
\begin{equation}\label{mainFlowRescaled}
\begin{aligned}
\hat{\gamma}_{\hat{t}}(\hat\theta,\hat t)
 &=  \phi'(t(\hat t))t'(\hat t)\big[\gamma(\theta(\hat\theta) ,t(\hat t))-\FP\big]
   + \phi(t(\hat t))(\partial_t\gamma)(\hat\theta,t(\hat t))t'(\hat t)
\\
&=  \phi(t(\hat t))\big[\gamma(\theta(\hat\theta) ,t(\hat t))-\FP\big]
   + \sqrt{2T}e^{-\hat t}(\sigma_1 k(\hat \theta,t(\hat t)) + \sigma_2)\nu(\hat\theta,t(\hat t))
\\
&= \hat{\gamma}(\hat\theta,\hat t) +
\left(\sigma_1\hat{k}(\hat\theta,\hat t) +
\sqrt{2T}e^{-\hat{t}}\sigma_2\right)\hat{\nu}(\hat\theta,\hat
t)\,.
\end{aligned}
\end{equation}
The motivation for choosing this particular rescaling is
that we want the enclosed area $\hat{A}(\hat{t})=\phi^2A(t)$ of
the normalised curve $\hat{\gamma}$ to approach a fixed value (in this case, $\sigma_1\pi$) as $\hat{t}$
approaches infinity.
To explain this we present the following fundamental lemma.

\begin{lem}
\label{LMareaconstrescaled}
Let $\hat{\gamma}:\S^1\times[0,\infty)\to \R^2$ be a solution to the rescaled
flow~\eqref{mainFlowRescaled} with convex initial data. Then
\[
\hat A(\hat\gamma(\cdot,\hat t)) \rightarrow \sigma_1\pi \text{ as }\hat t \nearrow \infty\,.
\]
\end{lem}
\begin{proof}
First note that we know $A(\gamma(\cdot,t(\hat t))\rightarrow 0$ and $L(\gamma(\cdot,t(\hat t))\rightarrow 0$ as $\hat t\rightarrow\infty$ by Theorem \ref{LAasympt}.
We calculate
\begin{align*}
\lim_{\hat t\rightarrow\infty} \hat A(\hat\gamma(\cdot,\hat t))
 &= \lim_{\hat t\rightarrow\infty} \phi^2(t(\hat t))A(\gamma(\cdot,t(\hat t)))
\\
 &= \lim_{t\rightarrow T} \frac{A(\gamma(\cdot,t))}{2T-2t}
\\
 &= \lim_{t\rightarrow T} \frac{-2\pi\omega\sigma_1 - \sigma_2 L(\gamma(\cdot,t))}{-2}
\,.
\end{align*}
Given that $L$ vanishes at final time (and $\omega = 1$), we conclude
\begin{align*}
\lim_{\hat t\rightarrow\infty} \hat A(\hat\gamma(\cdot,\hat t))
 &= \sigma_1\pi\,,
\end{align*}
as required.
\end{proof}

\begin{rmk}
In \cite{gageconvex}, the isoperimetric ratio is studied for curve shortening
flow and used to prove that $A(t)\rightarrow0$ implies $L(t)\rightarrow0$.
For the flow we study here this approach can not possibly work, since
$\sigma_2\ge0$ and
\[
\bigg(\frac{L^2(t)}{A(t)}\bigg)_t
 = - 2\sigma_1 \frac{L(t)}{A(t)}\left( \vn{k}_2^2 - \pi\frac{L(t)}{A(t)} \right)
   + 2\sigma_2 \frac{L(t)}{A(t)}\left( \frac12 \frac{L^2(t)}{A(t)} - 2\pi \right)
\,.
\]
To estimate the first term one may use the inequality $\vn{k}_2^2 \ge
\frac{4\pi^2}{L}$, which follows from the Poincar\'e inequality, combined with
the isoperimetric inequality.
The second term on the right is non-negative however and zero only for a circle (by the isoperimetric inequality).
\end{rmk}

We additionally note for interest that this rescaling and Theorem
\ref{LAasympt} give a bound on how fast $\vn{k}_2^2$ blows up for the convex
flow, effectively determining already the type of the singularity.

\begin{lem}
Let $\gamma:\S^1\times[0,T)\to\R^2$ be a family of closed, embedded, convex
plane curves evolving by the flow~\eqref{C-E problem}. Then
\[
\vn{k}_2^2(t) = O\bigg(\frac1{\sqrt{T-t}}\bigg)\,.
\]
\end{lem}
\begin{proof}
First, note that
\[
\limsup_{\hat t\rightarrow\infty} \hat L(\hat \gamma(\cdot,\hat t)) < \infty\,.
\]
This follows from the rescaled length estimate, Lemma \ref{rescaledlength}.
We do not give the proof of Lemma \ref{rescaledlength} here, but rather delay
it until we analyse the rescaled entropy, in particular its \emph{eventual
monotonicity}.
This is a crucial ingredient in the proof of Lemma \ref{rescaledlength}, and
takes quite some work to establish.

Therefore
\begin{align*}
\lim_{\hat t\rightarrow\infty} \hat L(\hat\gamma(\cdot,\hat t))
 &= \lim_{\hat t\rightarrow\infty} \phi(t(\hat t))L(\gamma(\cdot,t(\hat t)))
\\
 &= \lim_{t\rightarrow T} \frac{L(\gamma(\cdot,t))}{\sqrt{2T-2t}}
\\
 &= \lim_{t\rightarrow T} \frac{-\sigma_1\vn{k}_2^2(t) - 2\pi\sigma_2}{-\frac1{\sqrt{2T-2t} }}
\\
 &= \lim_{t\rightarrow T} \sigma_1\vn{k}_2^2(t)\sqrt{2T-2t}
\,.
\end{align*}
This proves the result.
\end{proof}

Our main tool is Huisken's monotonicity formula \cite{huisken1990asymptotic}, with a minor modification so that it is suitable for our flow.

\begin{thm}
\label{thmmonot}
Let $\hat{\gamma}:\S^1\times[0,\infty)\to \R^2$ be a solution to the rescaled
flow \eqref{mainFlowRescaled}.
Set
\[
\rho(\hat\gamma) = e^{-\frac{|\hat\gamma|^2}{2\sigma_1}}\,,\text{ and }
R(\hat t) = \int_{\hat\gamma} \rho\,d\hat s\,.
\]
Then
\begin{equation}
R'(\hat t)
= -\int_{\hat\gamma} Q^2\rho\,d\hat s
   + \frac{T\sigma_2^2}{2\sigma_1}e^{-2\hat t}
     \int_{\hat\gamma} \rho\,d\hat s
\,,
\label{EQmonot}
\end{equation}
where
\[
Q = \frac{\IP{\hat\gamma}{\hat\nu}}{\sqrt{\sigma_1}} + \sqrt{\sigma_1}\hat k + \frac{\sqrt{2T}\sigma_2}{2\sqrt{\sigma_1}}e^{-\hat t}\,.
\]
\end{thm}
\begin{proof}
We calculate
\begin{align*}
R'(\hat t)
&= \int_{\hat\gamma} \Big(
		1 - \sigma_1\hat k^2 - \sigma_2\hat k\sqrt{2T}e^{-\hat t} - \sigma_1^{-1}\IP{\hat\gamma}{\hat\gamma + (\sigma_1\hat k + \sigma_2\sqrt{2T}e^{-\hat t})\hat\nu}
        \Big)\rho\,d\hat s
\\&= 
   \int_{\hat\gamma} \Big(
		1 - \sigma_1\hat k^2 - \sigma_2\hat k\sqrt{2T}e^{-\hat t} - \sigma_1^{-1}|\hat\gamma|^2 - \sigma_1^{-1}\sigma_1\hat k\IP{\hat\gamma}{\hat\nu} - \sigma_1^{-1}\sigma_2\sqrt{2T}e^{-\hat t}\IP{\hat\gamma}{\hat\nu}
        \Big)\rho\,d\hat s
\\&= 
  -\int_{\hat\gamma} \left(
  \sqrt{\sigma_1^{-1}}\hat\gamma + \sqrt{\sigma_1}\hat k\hat\nu + \frac{\sqrt{2T}\sigma_2}{2\sqrt{\sigma_1}}e^{-\hat t}\hat\nu
        \right)^2\rho\,d\hat s
\\&\quad
   + \int_{\hat\gamma} \left(
   1 
   + \left(2\sqrt{\sigma_1^{-1}\sigma_1}-\sigma_1^{-1}\sigma_1\right)\hat k\IP{\hat\gamma}{\hat\nu} 
   + \left(\frac{1}{\sqrt{\sigma_1}^2} - \sigma_1^{-1}\right)\sigma_2\sqrt{2T}e^{-\hat t}\IP{\hat\gamma}{\hat\nu}
        \right)\rho\,d\hat s
\\&\qquad
   + \frac{T\sigma_2^2}{2\sigma_1}e^{-2\hat t}
     \int_{\hat\gamma} \rho\,d\hat s
\\&= 
  -\int_{\hat\gamma} \left(
  \frac1{\sqrt{\sigma_1}}\hat\gamma + \sqrt{\sigma_1}\hat k\hat\nu + \frac{\sqrt{2T}\sigma_2}{2\sqrt{\sigma_1}}e^{-\hat t}\hat\nu
        \right)^2\rho\,d\hat s
   + \int_{\hat\gamma} \Big(
   1 
   + \hat k\IP{\hat\gamma}{\hat\nu} 
        \Big)\rho\,d\hat s
\\&\qquad
   + \frac{T\sigma_2^2}{2\sigma_1}e^{-2\hat t}
     \int_{\hat\gamma} \rho\,d\hat s
     \,.
\end{align*}
Note that
\[
   \int_{\hat\gamma} (1 + \hat k\IP{\hat \gamma}{\hat \nu})\rho\,d\hat s
   = \int_{\hat\gamma} \sigma_1^{-1} \IP{\hat \gamma}{\hat \tau}^2\rho\,d\hat s
   \,,
\]
and
\begin{align*}
  \bigg(
  \frac{\hat \gamma-\IP{\hat \gamma}{\hat \tau}\hat \tau}{\sqrt{\sigma_1}} &+ \sqrt{\sigma_1}\hat k\hat \nu + \frac{\sqrt{2T}\sigma_2}{2\sqrt{\sigma_1}}e^{-\hat t}\hat \nu
  \bigg)^2
\\&
  = \bigg(
  \frac{\hat \gamma}{\sqrt{\sigma_1}} + \sqrt{\sigma_1}\hat k\hat \nu + \frac{\sqrt{2T}\sigma_2}{2\sqrt{\sigma_1}}e^{-\hat t}\hat \nu
  \bigg)^2
   - \sigma_1^{-1} \IP{\hat \gamma}{\hat \tau}^2
\,.
\end{align*}
This implies
\begin{align*}
R'(t)
&= 
  -\int_{\hat\gamma} \bigg(
  \frac{\IP{\hat\gamma}{\hat\nu}}{\sqrt{\sigma_1}} + \sqrt{\sigma_1}\hat k + \frac{\sqrt{2T}\sigma_2}{2\sqrt{\sigma_1}}e^{-\hat t}
        \bigg)^2\rho\,d\hat s
   + \frac{T\sigma_2^2}{2\sigma_1}e^{-2\hat t}
     \int_{\hat\gamma} \rho\,d\hat s
\\
&= -\int_{\hat\gamma} Q^2\rho\,d\hat s
   + \frac{T\sigma_2^2}{2\sigma_1}e^{-2\hat t}
     \int_{\hat\gamma} \rho\,d\hat s
\,,
\end{align*}
as required.
\end{proof}

A uniform bound for curvature follows by ideas of Chou-Zhu, see Appendix B.
This involves obtaining an a-priori estimate on the entropy (which also yields a length estimate, see Lemma \ref{rescaledlength}).
For the convenience of the reader we state the main estimate (Theorem \ref{TMrescaledcurvest}) here.

\begin{thm}
Let $\hat{\gamma}:\S^1\times[0,\infty)\to \R^2$ be a solution to the rescaled
flow \eqref{mainFlowRescaled} with convex initial data.
There exists a $\hat k_1\in[0,\infty)$ such that
	\[
		\hat k_{\max}(\hat t) \le \hat k_1
	\]
for all $\hat t \in [0,\infty)$.
\end{thm}

We can now finish our proof.

\begin{thm}
Let $\hat{\gamma}:\S^1\times[0,\infty)\to \R^2$ be a solution to the rescaled
flow \eqref{mainFlowRescaled} with convex initial data.
Then $\hat\gamma$ converges smoothly in the $C^\infty$ topology to a standard round circle.
\end{thm}
\begin{proof}
Integrating equation \eqref{EQmonot} we find
\[
R(0) - R(\hat t_j)
+ \frac{T\sigma_2^2}{2\sigma_1}\int_0^{\hat t_j} e^{-2\hat t}
     \int_{\hat\gamma} \rho\,d\hat s\,d\hat t
     = \int_0^{\hat t_j} \int_{\hat\gamma} Q^2\rho\,d\hat s\,d\hat t 
\,.
\]
By the evolution equation, $|\hat\gamma|$ converges, we let $\{\hat t_j\}$ be this subsequence and conclude that
\[
	\lim_{j\rightarrow\infty}
        \int_0^{\hat t_j} \int_{\hat\gamma} Q^2\rho\,d\hat s\,d\hat t 
\]
is bounded.
This implies that there is a subsequence that we also call $\{\hat t_j\}$ that along which we have
\[
	\int_{\hat\gamma} Q^2\rho\,d\hat s \rightarrow 0\,.
\]
Since we have a curvature bound and a length bound, the estimates from Section 3 apply to give uniform bounds on all derivatives of curvature.
As in Huisken's Proposition 3.4 of \cite{huisken1990asymptotic}, we obtain convergence to a solution of
\[
	\IP{\gamma}{\nu} = -\sigma_1k\,.
\]
Abresch-Langer \cite{abresch} have classified solutions to this equation.
In particular, the only embedded solution is a circle.
\end{proof}


\section{The non-convex case}


Without strict convexity we may not use the $\theta$ parametrisation.
We briefly calculate evolution equations for the rescaled flow in terms of rescaled arclength, beginning with $|\hat\gamma_{\hat u}|^2$:
\begin{align*}
|\hat{\gamma}_{\hat{u}}|_{\hat{t}} &=|\hat{\gamma}_{\hat{u}}|\IP{\partial_{\hat{s}}\left(\hat{\gamma}_{\hat{t}}\right)}{\hat{\tau}} \\
&=|\hat{\gamma}_{\hat{u}}|\IP{\left[\left(\sigma_1\hat{k}+\sqrt{2T}e^{-\hat{t}}\sigma_2\right)\hat{\nu}+\hat{\gamma}\right]_{\hat{s} }}{\hat{\tau}} \\
&=|\hat{\gamma}_{\hat{u}}|\left[\IP{-\hat{k}\left(\sigma_1\hat{k}+\sqrt{2T}e^{-\hat{t}}\sigma_2\right)\hat{\tau}}{\hat{\tau}}+\IP{\hat{\tau}}{\hat{\tau
}}\right] \\
&= |\hat{\gamma}_{\hat{u}}|\left[-\hat{k}\left(\sigma_1\hat{k}+\sqrt{2T}e^{-\hat{t}}\sigma_2\right)+1 \right].\numberthis{\label{C-E zhu rescaled speed evolution}}
\end{align*}

Therefore, the differential operators with respect to rescaled arc-length and rescaled time have the commutator relation
\begin{equation}\label{C-E commutator zhu rescaled}
\begin{aligned}
\partial_{\hat{t}}\partial_{\hat{s}} &= \partial_{\hat{t}}\left(|\hat{\gamma}_{\hat{u}}|^{-1}\partial_{\hat{u}}\right) = |\hat{\gamma}_{\hat{u}}|^{-1}\partial_{\hat{u}}\partial_{\hat{t}}-|\hat{\gamma}_{\hat{u}}|^{-2}|\hat{\gamma}_{\hat{u}}|_{\hat{t}}\,\partial_{\hat{u}}
\\
&= \partial_{\hat{s}}\partial_{\hat{t}} + \left[\hat{k}\left(\sigma_1\hat{k}+\sqrt{2T}e^{-\hat{t}}\sigma_2\right)-1 \right]\partial_{\hat{s}}.
\end{aligned}
\end{equation}
We use the definition of the rescaling to calculate
\begin{equation}\label{curvature zhu rescaled arclength}
\begin{aligned}
\hat{k}_{\hat{t}} &= \left(\frac{k}{\phi}\right)_{\hat{t}} = \left(\pD{(\phi^{-1})}{t}k+\phi^{-1}\pD{k}{t}\right)\pD{t}{\hat{t}} \\
&= \left(-\phi k + \phi^{-1}\left(\sigma_1k_{ss}+\sigma_1k^3+\sigma_2k^2\right)\right)\phi^{-2}
\\
&= \hat{k}^2\left(\sigma_1\hat{k}+\sqrt{2T}e^{-\hat{t}}\sigma_2\right)-\hat{k}+\sigma_1\hat{k}_{\hat{s}\hat{s}}\,.
\end{aligned}
\end{equation}
Then by \eqref{C-E commutator zhu rescaled} we have
\begin{equation}\label{C-E zhu rescaled k_st}
\begin{aligned}
\big(\hat{k}_{\hat{s}}\big)_{\hat{t}} &= \left[\hat{k}\left(\sigma_1\hat{k}+\sqrt{2T}e^{-\hat{t}}\sigma_2\right)-1\right]\hat{k}_{\hat{s}} 
	+ \partial_{\hat{s}}\big(\hat{k}_{\hat{t}}\big) \\
&= \sigma_1\big(\hat{k}_{\hat{s}}\big)_{\hat{s}\hat{s}}+\left(4\sigma_1\hat{k}^2 + 3\sqrt{2T}e^{-\hat{t}}\sigma_2\hat{k} -2\right)\hat{k}_{\hat{s}}\,.
\end{aligned}
\end{equation}
The key evolution equation we require is for $\vn{\hat k_{\hat s}}_2^2$.

\begin{lem}
\label{LMevol}
Let $\hat{\gamma}:\S^1\times[0,\infty)\to \R^2$ be a solution to the rescaled
flow \eqref{mainFlowRescaled}.
There exist continuous functions $\xi_1$, $\xi_2$ such that
\begin{align*}
\rD{}{\hat{t}}\int_{\hat{\gamma}} \hat{k}_{\hat{s}}^2\,d\hat{s} 
&= - 2\sigma_1\int_{\hat{\gamma}} \hat{k}_{\hat{s}\hat{s}}^2\,d\hat{s}
   - (1+\alpha)\int_{\hat{\gamma}}\hat{k}_{\hat{s}}^2 \,d\hat{s}
\\&\qquad
   + \bigg(
     7\sigma_1\hat{k}^2(\xi_1(\hat t))
   + 5\sqrt{2T}e^{-\hat{t}}\sigma_2\hat k(\xi_2(\hat t))
   - (2-\alpha)
   \bigg)\int_{\hat{\gamma}}\hat{k}_{\hat{s}}^2 \,d\hat{s}
\,,
\end{align*}
for any $\alpha\in\R$.
\end{lem}
\begin{proof}
We calculate
\begin{align*}
\rD{}{\hat{t}}\int_{\hat{\gamma}} \hat{k}_{\hat{s}}^2\,d\hat{s} 
&= \rD{}{\hat{t}}\int_{\hat{\gamma}} \hat{k}_{\hat{s}}^2|\hat{\gamma}_{\hat{u}}|\,d\hat{u} 
	= \int_{\hat{\gamma}} \big(\hat{k}_{\hat{s}}^2\big)_{\hat{t}}|\hat{\gamma}_{\hat{u}}|\,d\hat{u} +\int_{\hat{\gamma}} \hat{k}_{\hat{s}}^2\,|\hat{\gamma}_{\hat{u}}|_{\hat{t}}\,d\hat{u} 
\\
&= \int_{\hat{\gamma}} 2\hat{k}_{\hat{s}}\big(\hat{k}_{\hat{s}}\big)_{\hat{t}}\,d\hat{s} + \int_{\hat{\gamma}}\hat{k}_{\hat{s}}^2\left[-\hat{k}\left(\sigma_1\hat{k}+\sqrt{2T}e^{-\hat{t}}\sigma_2\right)+1 \right]\,d\hat{s}
\\
&= 2\sigma_1\int_{\hat{\gamma}}\hat{k}_{\hat{s}}\hat{k}_{\hat{s}\hat{s}\hat{s}}\,d\hat{s}+7\sigma_1\int_{\hat{\gamma}}\hat{k}^2\hat{k}_{\hat{s}}^2\,d\hat{s} + 5\sqrt{2T}e^{-\hat{t}}\sigma_2\int_{\hat{\gamma}}\hat{k}\hat{k}_{\hat{s}}^2 \,d\hat{s} - 3\int_{\hat{\gamma}}\hat{k}_{\hat{s}}^2 \,d\hat{s}
\\
&= - 2\sigma_1\int_{\hat{\gamma}} \hat{k}_{\hat{s}\hat{s}}^2\,d\hat{s}
   + 7\sigma_1\int_{\hat{\gamma}}\hat{k}^2\hat{k}_{\hat{s}}^2\,d\hat{s}
   + 5\sqrt{2T}e^{-\hat{t}}\sigma_2\int_{\hat{\gamma}}\hat{k}\hat{k}_{\hat{s}}^2 \,d\hat{s}
   - 3\int_{\hat{\gamma}}\hat{k}_{\hat{s}}^2 \,d\hat{s}
\,.
\end{align*}
Using the mean value theorem for integrals (see the appendix in \cite{shomberg} for example) we have
\begin{align*}
\rD{}{\hat{t}}\int_{\hat{\gamma}} \hat{k}_{\hat{s}}^2\,d\hat{s} 
&= - 2\sigma_1\int_{\hat{\gamma}} \hat{k}_{\hat{s}\hat{s}}^2\,d\hat{s}
   + 7\sigma_1\int_{\hat{\gamma}}\hat{k}^2\hat{k}_{\hat{s}}^2\,d\hat{s}
   + 5\sqrt{2T}e^{-\hat{t}}\sigma_2\int_{\hat{\gamma}}\hat{k}\hat{k}_{\hat{s}}^2 \,d\hat{s}
   - 3\int_{\hat{\gamma}}\hat{k}_{\hat{s}}^2 \,d\hat{s}
   \\
&= - 2\sigma_1\int_{\hat{\gamma}} \hat{k}_{\hat{s}\hat{s}}^2\,d\hat{s}
   - 3\int_{\hat{\gamma}}\hat{k}_{\hat{s}}^2 \,d\hat{s}
\\&\qquad
   + \bigg(
     7\sigma_1\hat{k}^2(\xi_1(\hat t))
   + 5\sqrt{2T}e^{-\hat{t}}\sigma_2\hat k(\xi_2(\hat t))
   \bigg)\int_{\hat{\gamma}}\hat{k}_{\hat{s}}^2 \,d\hat{s}
\,.
\end{align*}
Splitting up the second integral on the right hand side yields the lemma.
\end{proof}

We are now ready to prove our main result for this section.

\begin{thm}[Eventual convexity]
\label{TMcirclenonconvex}
Let $\hat{\gamma}:\S^1\times[0,\infty)\to \R^2$ be a solution to the rescaled
flow \eqref{mainFlowRescaled}.
Set
\[
T_{\text{max}} := \frac{1}{2\pi\sigma_1\sigma_2}\min\{L_0\sigma_1, A_0\sigma_2\}
\,,
\]
where $L_0$ and $A_0$ are the initial length and enclosed area of the original (not rescaled) flow.
Suppose that the initial data of the original flow satisfies
\begin{equation}\label{nonconvexbound2}
\sqrt{L_0}\vn{k_{s}}_2\Big|_{t=0}
\le 
   \frac{1}{14\sigma_1}\left(
   \sqrt{
	25\sigma_2^2
	+ \frac{14\sigma_1(2-\alpha)}{\Tm}
	}
 - 5\sigma_2
	\right)
\,,
\end{equation}
for some $\alpha\in(0,2)$.
Then the rescaled flow converges to a standard round circle as $\hat t\rightarrow\infty$.
\end{thm}
\begin{proof}
We make the following assumption:
\begin{equation}
\label{ass}
\text{For all $\hat t$ the function $\hat k(\cdot,\hat t)$ has a zero.}
\end{equation}
Our goal in this proof is to contradict \eqref{ass}.
Assumption \eqref{ass} implies the estimate 
\[
	\hat k \le \int_{\hat\gamma} |\hat k_{\hat s}|\,d\hat s\,.
\]
Therefore
\begin{align*}
	7\sigma_1\hat{k}^2 + 5\sqrt{2T}e^{-\hat{t}}\sigma_2\hat k
	&\le 7\sigma_1{\hat L}\vn{\hat k_{\hat s}}_2^2 + 5\sqrt{2T}\sigma_2e^{-\hat t}{\sqrt{\hat L}}\vn{\hat k_{\hat s}}_2
	\\
	&\le {\sqrt{\hat L}}\vn{\hat k_{\hat s}}_2
	\left(7\sigma_1{\sqrt{\hat L}}\vn{\hat k_{\hat s}}_2
	+ 5\sigma_2\sqrt{2T} e^{-\hat t}\right)
\end{align*}
Setting $x = {\sqrt{\hat L}}\vn{\hat k_{\hat s}}_2$ we require $7\sigma_1 x^2 + 5\sigma_2e^{-\hat t}\sqrt{2T} x < 2-\alpha$.
Since we have no control over how large $\hat t$ is, we rephrase this as $7\sigma_1 x^2 + 5\sigma_2\sqrt{2T} x < 2-\alpha$.
As $x\ge0$, the condition we require is
\[
14\sigma_1{\sqrt{\hat L}}\vn{\hat k_{\hat s}}_2
\le 
   \sqrt{
	50T\sigma_2^2
	+ 28\sigma_1(2-\alpha)
	}
 - 5\sigma_2\sqrt{2T}
\,.
\]
The function $b\mapsto -b + \sqrt{b^2 + 28\sigma_1(2-\alpha)}$ is decreasing, and so our condition is
\begin{equation*}
14\sigma_1{\sqrt{\hat L}}\vn{\hat k_{\hat s}}_2
\le 
   \sqrt{
	50\Tm\sigma_2^2
	+ 28\sigma_1(2-\alpha)
	}
 - 5\sigma_2\sqrt{2\Tm}
\,.
\end{equation*}
Under this condition $\vn{\hat k_{\hat s}}_2^2$ is initially decreasing exponentially fast.
However we need to satisfy \eqref{nonconvexbound2} for positive $\hat t$ so that exponential decay continues.
This is difficult because of the rescaled length factor.
Set
\[
\tilde L(\hat t) = e^{-\hat t}\hat L(\hat t)\,.
\]
Then $\tilde L' \le 0$, in particular,
\[
\tilde L(\hat t) \le \tilde L(0) - \int_0^{\hat t}\sigma_1\vn{\hat k}_2^2 + 2\pi\sigma_2\sqrt{2T} e^{-\hat t}\,d\hat t\,.
\]
While the bound \eqref{nonconvexbound2} is satisfied, we have
\begin{align*}
\rD{}{\hat{t}}\bigg(\tilde L\int_{\hat{\gamma}} \hat{k}_{\hat{s}}^2\,d\hat{s} \bigg)
&\le - (1+\alpha)\tilde L\int_{\hat{\gamma}}\hat{k}_{\hat{s}}^2 \,d\hat{s}
\,,
\end{align*}
or
\begin{align*}
\rD{}{\hat{t}}\bigg(e^{(1+\alpha)\hat t}\tilde L\int_{\hat{\gamma}} \hat{k}_{\hat{s}}^2\,d\hat{s} \bigg)
&\le 0
\,.
\end{align*}
This gives the decay estimate
\[
\hat L\int_{\hat{\gamma}} \hat{k}_{\hat{s}}^2\,d\hat{s}
\le \hat L(0)\vn{\hat k_{\hat s}}_2^2(0)e^{-\alpha \hat t}\,.
\]
Now there are only two remaining possibilities.
Either $\vn{\hat k_{\hat s}}_2^2$ or $\hat L$ converges to zero.

{\bf Case 1.} {\it $\vn{\hat k_{\hat s}}_2^2$ converges to zero.}
This is a contradiction, since it would imply $\hat k$ is asymptotic to a constant in $L^2$, however as $\hat\gamma$ is closed, this constant has to be strictly positive.
But then for $\hat t$ sufficiently large the function $\hat k$ has no zeroes; a contradiction with assumption \eqref{ass}.

{\bf Case 2.} {\it $\hat L$ converges to zero.}
Note that $L(t)$ also converges to zero, since $\tilde L(\hat t) = e^{-\hat t}\hat L(\hat t)$.
Then the calculation in Lemma \ref{LMareaconstrescaled} applies to show that $\hat A \rightarrow \sigma_1\pi > 0$.
But this contradicts the isoperimetric inequality.

Therefore in any case we must have at some time strict convexity of the flow.
After this time it remains so, and the theory in the first part of the paper
applies to yield smooth convergence to a standard round circle.
\end{proof}


\section{A Lifespan theorem}

In this section we estimate the maximal time from below.
We follow the literature (primarily for higher-order flows, see the original
work of Kuwert-Sch\"atzle \cite{kuwert}, which was later extended to other settings \cite{wh2,wh1,wh3,wh4,wh5,wh6}).
The idea for this kind of estimate can be traced back to work of Struwe on the
harmonic map heat flow \cite{struwe}.
Each new setting requires different integral estimates; in the setting we study
here, we use for the first time a product of two functionals (for scale-invariance reasons) and new arguments to contradict
$T=\lambda$.
This is because the standard approach (that this implies the flow may be
smoothly extended) may not be strictly true.
Indeed, it is a reasonable possibility that the length would contract to zero
so quickly that the quantity $L\vn{k}_2^2$ remains small while the solution
shrinks to a point.
In order to rule this out, we use an argument relying on isoperimetry and the
continuous rescaling from Section 4.

We use the notation
\[
	L_\phi = \int_\gamma \phi\,ds\,,\quad\text{and}\quad L_{B_\rho(x)} = \int_{\gamma^{-1}(B_\rho(x))}\,ds
\]
to denote localised length.
In the equation above, $\phi = \tilde\phi\circ\gamma$, with
$\tilde\phi:\R^2\rightarrow[0,1]$ a smooth function such that
\[
\chi_{B_{\rho/2}(x)} \le \tilde\phi \le \chi_{B_\rho(x)}
\]
and
\[
|\phi_s| \le \frac{c}{\rho}\,.
\]
To indicate that we are using $\phi$ as above, we say simply that $\phi$ is a
smooth cutoff function on balls of radii $\rho$ and $\rho/2$.

We warm up with the following lemma, that contains a key integral estimate.
\begin{lem}
\label{LMest}
Let $\gamma:\S^1\times[0,T)\to\R^2$ be a solution to \eqref{C-E problem}.
	Let $T^*\in[0,T)$, $\rho>0$ be given.
Suppose
\begin{equation}
\label{EQsmallness}
	\sup_{t\in[0,T^*]}\sup_{x\in\R^2}L_{B_\rho(x)}\int_{\gamma^{-1}(B_\rho(x))} k^2\,ds
	\le \varepsilon \le \frac1{16}
	\,.
\end{equation}
Suppose additionally that $\gamma$ is not convex on $[0,T^*]$, that is, that
$k(\cdot,t):\S^1\rightarrow\R$ has a zero for each $t\in[0,T^*]$.
Then for all $t\in[0,T^*]$ we have
\begin{align*}
	L_{\phi^4}\int_\gamma k^4\phi^4\,ds
	&\le \frac12 L_{\phi^4}\int_\gamma k_s^2 \phi^4\,ds
	+ \frac{c\varepsilon^2}{\rho^2}\,.
\end{align*}
for an absolute constant $c>0$, and where $\phi$ is a smooth cutoff function on balls of radii $\rho$ and $\rho/2$ centred anywhere.
\end{lem}
\begin{proof}
For brevity let $\psi := \phi^4$.
We calculate
\begin{align*}
	L_\psi\int_\gamma k^4\psi\,ds
	&= (k^2\phi^{2})(\xi(t))L_\psi\int_\gamma k^2\phi^2 \,ds
	\\
	&\le 2L_\psi\int_\gamma k^2\phi^2 \,ds\int_\gamma kk_s \phi^2\,ds
	+ 2L_\psi\int_\gamma k^2\phi^2 \,ds\int_\gamma k^2 \phi\phi_s\,ds
	\\
	&\le 2L_\psi\int_\gamma k^2\phi^2 \,ds\bigg(\int_\gamma k^2\chi_{B_{\rho(x)}}\,ds\int_\gamma k_s^2 \psi\,ds\bigg)^\frac12
	+ \frac{cL_\psi}{\rho}\int_\gamma k^2\phi^2 \,ds\int_\gamma k^2 \phi\,ds
	\\
	&\le \delta L_\psi\int_\gamma k_s^2 \psi\,ds
	+ \frac{L_\psi}{\delta}\bigg(\int_\gamma k^2\phi^2 \,ds\bigg)^2\int_\gamma k^2\chi_{B_{\rho(x)}}\,ds
	\\&\qquad
	+ \delta\bigg(\int_\gamma k^2\phi^2 \,ds\bigg)^2
	+ \frac{cL_\psi^2}{\delta\rho^2}\bigg(\int_\gamma k^2 \phi\,ds\bigg)^2
	\\
	&\le \delta L_\psi\int_\gamma k_s^2 \psi\,ds
	+ \frac{L_\psi^2}{\delta}\int_\gamma k^4\psi \,ds\int_\gamma k^2\chi_{B_{\rho(x)}}\,ds
	\\&\qquad
	+ \delta L_\psi\int_\gamma k^4\psi \,ds
	+ \frac{c\varepsilon^2}{\delta\rho^2}
	\\
	&\le \delta L_\psi\int_\gamma k_s^2 \psi\,ds
	+ \Big(\delta + \varepsilon\frac{1}{\delta}\Big)L_\psi\int_\gamma k^4\psi \,ds
	+ \frac{c\varepsilon^2}{\delta\rho^2}
	\,.
\end{align*}
Let us take $\delta = \frac14$ and assume $\varepsilon \le \frac1{16}$; then absorb to find
\begin{align*}
	L_\psi\int_\gamma k^4\psi\,ds
	&\le \frac12 L_\psi\int_\gamma k_s^2 \psi\,ds
	+ \frac{c\varepsilon^2}{\rho^2}\,,
\end{align*}
as required.
\end{proof}

Now we prove the key integral estimate that controls the concentration of curvature along the flow.
\begin{prop}
\label{PNpreservesmallness}
Let $\gamma:\S^1\times[0,T)\to\R^2$ be a solution to \eqref{C-E problem}.
Let $T^*\in[0,T)$, $\rho>0$ be given.
Suppose that $\gamma$ is not convex on $[0,T^*]$, that is, that
$k(\cdot,t):\S^1\rightarrow\R$ has a zero for each $t\in[0,T^*]$.
There exists an absolute constant $\varepsilon_0>0$ such that if
\begin{equation*}
	\sup_{t\in[0,T^*]}\sup_{x\in\R^2}L_{B_\rho(x)}\int_{\gamma^{-1}(B_\rho(x))} k^2\,ds
	= \varepsilon \le \varepsilon_0 \le \frac1{16}
	\,,
\end{equation*}
then
\begin{equation}
\label{EQtheest}
	L_{\phi^4}\int_{\gamma^{-1}(B_\frac\rho2(x))}k^2\,ds
	\le 
	\bigg(L_{\phi^4}\int_{\gamma^{-1}(B_\rho(x))}k^2\,ds\bigg)\bigg|_{t=0}
	+ c_0t(1+\rho^{-2})\varepsilon
	 \,,\quad t\in[0,T^*]\,,
\end{equation}
for a constant $c_0>0$ depending only on $\sigma_1$ and $\sigma_2$.
\end{prop}
\begin{proof}
We again let $\psi := \phi^4$.
Differentiating,
\begin{align*}
	\rD{}{t}\bigg(L_{\psi}\int_\gamma k^2 \psi\,ds\bigg)
	&=
	\int_\gamma k^2 \psi\,ds\bigg(
			   \int_\gamma (-\sigma_1k^2 -\sigma_2k)\psi\,ds
			+ 4\int_\gamma \phi_t\phi^3\,ds
                           \bigg)
	\\&\qquad
	+ L_\psi\int_\gamma 4k^2 \phi^3(D\tilde\phi|_\gamma\cdot\nu)(\sigma_1k + \sigma_2)\,ds
	\\&\qquad
	+ L_\psi\int_\gamma (2k(\sigma_1k_{ss} + k^2(\sigma_1k+\sigma_2))-\sigma_1k^4 -\sigma_2k^3) \psi\,ds
	\\
	&= - 2L_\psi\sigma_1\int_\gamma k_s^2\,\psi\,ds
	   - 8L_\psi\sigma_1\int_\gamma k_sk\,\phi^3\phi_s\,ds
	\\&\qquad
	   - \sigma_1\bigg(\int_\gamma k^2 \psi\,ds\bigg)^2
	   - \sigma_2\int_\gamma k^2 \psi\,ds\int_\gamma k \psi\,ds
	\\&\qquad
	+ L_\psi\int_\gamma (\sigma_1k^4 + \sigma_2k^3) \psi\,ds
	\\&\qquad
	+ 4L_\psi\int_\gamma k^2 \phi^3(D\tilde\phi|_\gamma\cdot\nu)(\sigma_1k + \sigma_2)\,ds
	\\&\qquad
			+ 4\int_\gamma k^2\psi\,ds \int_\gamma (D\tilde\phi|_\gamma\cdot\nu)(\sigma_1k + \sigma_2)\phi^3\,ds
	\\
	&\le - 2L_\psi\sigma_1\int_\gamma k_s^2\,\psi\,ds
	   - \sigma_1\bigg(\int_\gamma k^2 \psi\,ds\bigg)^2
	   - \sigma_2\int_\gamma k^2 \psi\,ds\int_\gamma k \psi\,ds
	\\&\qquad
	   - 8L_\psi\sigma_1\int_\gamma k_sk\,\phi^3\phi_s\,ds
	+ L_\psi\int_\gamma (\sigma_1k^4 + \sigma_2k^3) \psi\,ds
	\\&\qquad
	+ 4L_\psi\int_\gamma k^2 \phi^3(D\tilde\phi|_\gamma\cdot\nu)(\sigma_1k + \sigma_2)\,ds
	\\&\qquad
			+ 4\int_\gamma k^2\psi\,ds \int_\gamma (D\tilde\phi|_\gamma\cdot\nu)(\sigma_1k + \sigma_2)\phi^3\,ds
	\,.
\end{align*}
We observe the estimates
\[
	   - 8L_\psi\sigma_1\int_\gamma k_sk\,\phi^3\phi_s\,ds
	   \le
	       4\delta L_\psi\sigma_1\int_\gamma k_s^2\,\psi\,ds
	       + \frac{4\sigma_1c^2}{\rho^2\delta}L_\psi\int_\gamma k^2\,\phi^2\,ds
	       \,,
\]
\begin{align*}
	L_\psi\int_\gamma 4k^2 \phi^3(D\tilde\phi|_\gamma\cdot\nu)(\sigma_1k + \sigma_2)\,ds
	&\le \frac{c\sigma_1}{\rho}L_\psi\int_\gamma k^3 \phi^3\,ds
	+ \frac{c\sigma_2}{\rho} L_\psi\int_\gamma k^2\,\phi^3\,ds
       \\
       &\le \frac{\sigma_1}{2}L_\psi\int_\gamma k^4 \psi\,ds
	+ \frac{c\sigma_1}{\rho^2}
	       L_\psi\int_\gamma k^2\,\phi^2\,ds
\\&\qquad
	+ \frac{c\sigma_2}{\rho^2} L_\psi\int_\gamma k^2\,\phi^2\,ds
	+ c\sigma_2L_\psi\int_\gamma k^2\,\psi\,ds
	\,,
\end{align*}
\begin{align*}
	4\int_\gamma k^2\psi\,ds &\int_\gamma (D\tilde\phi|_\gamma\cdot\nu)(\sigma_1k + \sigma_2)\phi^3\,ds
	   - \sigma_2\int_\gamma k^2 \psi\,ds\int_\gamma k \psi\,ds
\\
	&\le
		\bigg(\frac{c\sigma_1}{\rho} + \sigma_2\bigg)
	 L_{\phi^2}^\frac12 \bigg(\int_\gamma k^2\psi\,ds\bigg)^\frac32
	+ \frac{c\sigma_2}{\rho}L_{\phi^3}\int_\gamma k^2\psi\,ds
\\&\le
	\frac{\sigma_1}{4}\bigg(\int_\gamma k^2\psi\,ds\bigg)^2
	+ c(\sigma_1\rho^{-2}+\sigma_2^2\sigma_1^{-1} + \sigma_2\rho^{-1})\varepsilon
\\&\le
	\frac{\sigma_1}{4} L_\psi\int_\gamma k^4\psi\,ds
	+ c(\sigma_1\rho^{-2}+\sigma_2^2\sigma_1^{-1})\varepsilon
	\,,
\end{align*}

\[
	L_\psi\sigma_2\int_\gamma k^3\psi\,ds
	\le \frac{\sigma_1}{4}L_\psi\int_\gamma k^4 \psi\,ds
	+ \frac{\sigma_2^2}{\sigma_1}
	       L_\psi\int_\gamma k^2\,\psi\,ds
	\,,
\]
\[
	2L_\psi\sigma_1\int_\gamma k^4\psi\,ds
	\le  L_\psi\sigma_1\int_\gamma k_s^2 \psi\,ds
	+ \frac{c\sigma_1\varepsilon^2}{\rho^2}
	\,.
\]
The last estimate follows from Lemma \ref{LMest}.
Therefore we have
\begin{align*}
	\rD{}{t}\bigg(L_\psi\int_\gamma k^2 \psi\,ds\bigg)
	&\le -  L_\psi\sigma_1\int_\gamma k_s^2\,\psi\,ds
	   - \sigma_1\bigg(\int_\gamma k^2 \psi\,ds\bigg)^2
	+ c_0(1+\rho^{-2})
	       \varepsilon
	       \,,
\end{align*}
where $c_0$ is a constant depending only on $\sigma_1$, $\sigma_2$ and $\tilde\phi$.
Integrating, we find
\begin{align*}
	L_\psi\int_\gamma k^2 \psi\,ds
	&+ \int_0^t \bigg(L_\psi\sigma_1\int_\gamma k_s^2\,\psi\,ds
	   + \sigma_1\bigg(\int_\gamma k^2 \psi\,ds\bigg)^2
	   \bigg)\,dt'
	   \\
	&\le
	L_\psi(0)\int_\gamma k^2 \psi\,ds\bigg|_{t=0}
	 + c_0(1+\rho^{-2})t\varepsilon
	       \,.
\end{align*}
This estimate implies \eqref{EQtheest}.
\end{proof}

\begin{thm}[Lifespan theorem]
Let $\gamma:\S^1\times[0,T)\to\R^2$ be a non-convex solution to \eqref{C-E problem}.
	There are constants $\rho\in(0,1)$, $\varepsilon_1>0$, and $c_0<\infty$ such that
\begin{equation*}
	\sup_{x\in\R^2}L_{B_\rho(x)}\int_{\gamma^{-1}(B_\rho(x))} k^2\,ds\Big|_{t=0} = \varepsilon(x) \le \varepsilon_1
\end{equation*}
implies that the maximal time $T$ satisfies
\begin{equation*}
T \ge \frac{1}{c_0}\rho^2\,,
\end{equation*}
and we have the estimate
\begin{equation}
\label{EQltestnear}
	L_{B_{\frac{\rho}{2}}(x)}\int_{\gamma^{-1}(B_\frac\rho2(x))}k^2\,ds
	\le 
	c\varepsilon_1
\qquad\text{ for }\qquad
0\le t \le \frac{1}{c_0}\rho^2.
\end{equation}
\end{thm}
\begin{proof}
We make the definition
\begin{equation}
\label{e:epsfuncdef}
\eta(t) = \sup_{x\in\R^2}L_{B_\rho(x)}\int_{\gamma^{-1}(B_\rho(x))} k^2\,ds\,.
\end{equation}
By covering $B_\rho(x) \subset \R^2$ with several translated copies of
$B_{\frac{\rho}{2}}$ there is a universal constant $c_{\eta}$ such that
\begin{equation}
\label{e:epscovered}
\eta(t) \le c_{\eta}\sup_{x\in\R^2}L_{B_{\frac{\rho}{2}}(x)}\int_{\gamma^{-1}(B_{\frac{\rho}{2}}(x))} k^2\,ds\,.
\end{equation}
By short time existence the function $\eta:[0,T)\rightarrow\R$ is continuous.  We now define
\begin{equation}
\label{e:struweparameter}
t_0
=
  \sup\{0\le t\le\min(T,\lambda) : \eta(\tau)\le \delta
                                         \ \text{ for }\ 0\le\tau\le t\},
\end{equation}
where $\lambda$, $\delta$ are parameters to be specified later.

The proof continues in three steps.
\begin{align}
\label{e:7}
t_0 &= \min(T,\lambda),\\
\label{e:8}
t_0 &= \lambda \quad\Longrightarrow\quad\text{Lifespan theorem},\\
\label{EQfinal}
T &\ne \infty\hskip+2.42mm \Longrightarrow\quad t_0 \ne T\,.
\end{align}
The three statements \eqref{e:7}, \eqref{e:8}, \eqref{EQfinal} together imply the Lifespan theorem.
The argument is as follows: first notice that by \eqref{e:7}
$t_0 = \lambda$ or $t_0 = T$, and if $t_0 = \lambda$ then by \eqref{e:8} we
have the Lifespan theorem.  Also notice that if $t_0 = \infty$ then $T = \infty$ and the
Lifespan theorem follows from estimate \eqref{EQltproof} below (used to prove statement
\eqref{e:8}).  Therefore the only remaining case where the Lifespan theorem may fail to be true is
when $t_0 = T < \infty$.  But this is impossible by statement \eqref{EQfinal}, so we are finished.

To prove step 1, suppose it is false.
This means that $t_0 < \lambda$.
This implies that on $[0,t_0)$ we have $\eta(t) \le \delta$, and
$\eta(t_0) = \delta$.
Setting $\tilde\phi$ to be a cutoff function that is identically one on
$B_{\sfrac\rho2}(x)$ and zero outside $B_\rho(x)$, so that $\phi = \tilde\phi
\circ \gamma$ has the corresponding properties on the preimages of these balls
under $\gamma$, Proposition \ref{PNpreservesmallness} implies
\[
	L_{B_\frac\rho2(x)}\int_{\gamma^{-1}(B_\frac\rho2(x))}k^2\,ds
	\le 
	\bigg(L_{B_\rho(x)}\int_{\gamma^{-1}(B_\rho(x))}k^2\,ds\bigg)\bigg|_{t=0}
	+ c_0(1+\rho^{-2})t\varepsilon_1
	 \,,\quad t\in[0,t_0)\,.
\]
We use the assumption $\rho \le 1$ to conclude
\begin{equation}
\label{EQltproof0}
	L_{B_\frac\rho2(x)}\int_{\gamma^{-1}(B_\frac\rho2(x))}k^2\,ds
	\le 
	\bigg(L_{B_\rho(x)}\int_{\gamma^{-1}(B_\rho(x))}k^2\,ds\bigg)\bigg|_{t=0}
	+ 2c_0\rho^{-2}t\varepsilon_1
	 \,,\quad t\in[0,t_0)\,.
\end{equation}
The aforementioned covering argument yields
\begin{equation}
	L_{B_\rho(x)}\int_{\gamma^{-1}(B_\rho(x))}k^2\,ds
	\le 
	c_\eta \varepsilon_1
	+ 2c_\eta c_0\rho^{-2}\lambda\varepsilon_1
	 \,.
\label{EQltproof}
\end{equation}
We choose $\delta = 3c_\eta \varepsilon_1$, and
$\varepsilon_1$ small enough such that $\delta \le \varepsilon_0$ where
$\varepsilon_0$ is the constant from Proposition \ref{PNpreservesmallness}.
Then, picking $\lambda = \rho^2/c_0$, the estimate \eqref{EQltproof} implies
\[
	3c_\eta \varepsilon_1 = \eta(t) < 3c_\eta \varepsilon_1
	\,,\quad\text{ for all }0\le t \le t_0\,,
\]
which is a contradiction.

We have also proved the second step \eqref{e:8}. Observe that if $t_0 =
\lambda$ then by the definition \eqref{e:struweparameter} of $t_0$, 
\[
  T\ge\lambda=\rho^2/c_0\,,
\]
which is the lower bound for maximal time claimed by the Lifespan theorem.
That is, we have proved if $t_0 = \lambda$, then the Lifespan theorem holds, which is the second step.
The estimate \eqref{EQltestnear} follows from $\eta(t) \le 3c_\eta\varepsilon_1$ above, valid for $t\in[0,t_0]$.

We assume
\[
t_0=T\ne\infty;
\]
since if $T=\infty$ then the lower bound on $T$ holds automatically and again the previous estimates imply
the a-priori control \eqref{EQltproof0} on $L_{B_\frac\rho2(x)}\vn{k}_{2,\gamma^{-1}(B_{\frac\rho2}(x))}^2$.
Note also that we can safely assume $T < \lambda$, since otherwise
we can apply step two to conclude the Lifespan theorem.

We now show that this can only lead to a contradiction.
Suppose first that $L$ is asymptotic to $\mu > 0$. Then by hypothesis $\vn{k}_2^2$
is uniformly bounded and by a compactness theorem (see e.g. \cite{breuning})
the flow has a $C^{1,\alpha}$ limit as $t\rightarrow T$.
Local existence and uniqueness for such initial data (by using an argument
analogous to Deckelnick \cite{deckelnick} for example) we may extend the flow
past $T$, a contradiction.

The only remaining possibility is that $L$ is asymptotic to zero.
Note that the isoperimetric inequality implies that $A$ is also asymptotic to zero.
Now consider the rescaled flow $\hat\gamma$ as in Section 5.
Note that the proof of Lemma \ref{LMareaconstrescaled} works in this case without change, and implies that
\begin{equation}
	\label{EQlast}
	\lim_{\hat t\rightarrow\infty} \hat A = \sigma_1\pi
	\,.
\end{equation}
Since the area and length of $\gamma$ approach zero, there is a time $t^*$ such that
for all $t>t^*$,
\[
	L\vn{k}_2^2(t) = \sup_{x\in\R^2} L_{B_\frac\rho2(x)}\vn{k}_{2,\gamma^{-1}(B_{\frac\rho2}(x))}^2(t)\,.
\]
Therefore, for $t>t^*$ we have the uniform estimate $L\vn{k}_2^2 \le
\varepsilon$.
We now calculate for the rescaled flow
\begin{align*}
	\lim_{\hat t\rightarrow \infty}\hat L^2 
	&= 
	\lim_{t\rightarrow T} \frac{L^2}{2T - 2t}
	\\
	&= 
	\lim_{t\rightarrow T} -L(-\sigma_1\vn{k}_2^2 -2\pi\sigma_2)
	\\
	&= 
	\sigma_1\lim_{t\rightarrow T} L\vn{k}_2^2
\end{align*}
so that $\hat L^2(\hat t)
\le \sigma_1\varepsilon + c_1(\hat t)$, where $c_1(\hat t)\rightarrow 0$ as $\hat t\rightarrow\infty$.
The isoperimetric inequality implies $\hat L^2 \ge 4\pi\hat A$, so that
by \eqref{EQlast} we have
$\hat L^2(\hat t) \ge 4\sigma_1\pi^2 - c_2(\hat t)$,
where $c_2(\hat t)\rightarrow0$ as $\hat t\rightarrow\infty$.
This is a contradiction for $\hat t$ sufficiently large, so long as $\varepsilon <
4\pi^2$ (which is an absolute constant).

This establishes \eqref{EQfinal} and finishes the proof of the theorem. 
\end{proof}

\section{Application of the Lifespan theorem to blowups}

In this section we outline how to prove Theorem \ref{TMlifespanappl}.
The key assumption is that
\begin{equation}
\label{EQass}
\text{Suppose that $\gamma_t$ do not contract to a round point as $t\nearrow T$.}
\end{equation}
The first dot point follows because if $\gamma_t$ were convex at any time, then
Theorem \ref{maintheorem} would apply with $\gamma_t$ as initial data and the
flow would contract to a round point. This is in contradiction with \eqref{EQass}.

The second dot point follows because if at any time that inequality is violated
then the second part of Theorem \ref{maintheorem} applies and we again have
contraction to a round point.

For the third dot point, our idea is as follows.
We use the Lifespan theorem \ref{TMmaximaltimeestimates} at each $t$ to
rescale the solution, producing a sequence that converges to a limiting flow.
This limiting flow is called a \emph{blowup}.
We calculate that it solves the curve shortening flow and is ancient.
To this flow we apply a theorem of Daskalopoulos-Hamilton-Sesum \cite{DHS},
that gives our desired partial classification (the third dot point).

Let us define the critical radius
\[
r_t = \sup\left\{\rho>0\,:\, \forall x\in\R^2\,, L_{B_\rho(x)}\int_{\gamma^{-1}(B_\rho(x))} k^2\,ds\le\varepsilon_1 \right\}
\]
for $t\in[0,T)$.
Note that
\[
L_{B_{r_t}(x)}\int_{\gamma^{-1}(B_{r_t}(x))} k^2\,ds\le\varepsilon_1\qquad \forall x\in\R^2\,
\]
and that there exists for each $t\in[0,T)$ an $x_t\in\R^2$ such that
\[
L_{\overline{B_{r_t}(x_t)}}\int_{\gamma^{-1}(\overline{B_{r_t}(x_t))}} k^2\,ds\ge\varepsilon_1\,.
\]
The Lifespan Theorem implies that (for $T$ the maximal existence time)
\[
c_0^{-1}r_t^2 \le T-t\,,\text{ for }0 \le t < T\,.
\]
In particular this implies $r_t\rightarrow0$ as $t\nearrow T$ (for $T<\infty$).
Let us now introduce the notation $\gamma_t(\cdot) := \gamma(\cdot, t)$.
Consider the \emph{discrete} rescaling
\[
\gamma_v^t := \frac{1}{r_t} (\gamma_{t + vr_t^2} - x_t)
\]
for $0 \le v \le c_0$.
Note that $\gamma_v^t:\S\times[-r_t^{-2}t,r_t^{-2}(T-t))\rightarrow\R^2$,
and
\begin{equation}
\label{EQlifespanrescaling}
\partial_v\gamma_v^t = (\sigma_1 r_tk_{t+vr_t^2} + \sigma_2r_t)\nu_{t+vr_t^2}
 = (\sigma_1 k_{\gamma_v^t} + \sigma_2r_t)\nu_{\gamma_v^t}
\,.
\end{equation}

Then
\[
L_{B_1(x)}^{\gamma_v^t}\int_{(\gamma_v^t)^{-1}({B_1(x))}} k_{\gamma_v^t}^2\,ds_{\gamma_v^t}
=
L_{B_{r_t}(x)}^{\gamma_t}\int_{\gamma_t^{-1}({B_{r_t}(x))}} k_{\gamma_t}^2\,ds_{\gamma_t}
\le c\varepsilon_1
\,,
\]
for $-r_t^{-2}T \le v \le c_0$, $0\le t \le T$, $x\in\R^2$.
A standard argument yields uniform estimates for all derivatives of curvature for the flows $\gamma_v^t$ where $v\le c_0$.
Therefore a compactness theorem applies and we deduce the existence of
subsequences $t_j\rightarrow T$ such that the flows
$(\gamma^{t_j}_v)_{v\in[-r_{t_j}^{-2}T,c_0]}$ converge after reparametrisation smoothly on compact subsets of $\R^2$
to a flow $(\gamma^\infty_v)_{v\in(-\infty,c_0]}:\S^\infty\rightarrow\R^2$, where
$\S^\infty$ is an open 1-manifold without boundary, possibly not connected.
If $\S^\infty$ contains a compact component, then $\S^\infty$ is equal to this
component, that is, has no further components.

We compute that the blowup $\gamma^\infty$ satisfies for $v\in(-\infty,c_0]$ the evolution equation
\[
\partial_v \gamma^\infty_v = \sigma_1 k_{\gamma_v^t}\nu_{\gamma_v^t}\,.
\]
The blowup is therefore an ancient non-convex solution to the curve shortening flow.
Note that for $v\in(-\infty,0)$ curvature is uniformly bounded and we have
$\vn{k}_1 \le c\varepsilon_1$.
If the blowup is embedded and compact, Daskalopoulos-Hamilton-Sesum \cite{DHS} implies that the blowup must be convex.
Therefore each component of the blowup must be either noncompact, nonembedded, or both.

\appendix

\section{Estimates for the convex flow}
\label{sec:C-E convergence}

This section follows ideas of Chou-Zhu for anisotropic flows by curvature.
We produce the details here in our special case for the convenience of the reader.

We shall investigate the asymptotic behaviour of the flow with the
key assumption being convexity. We claim that a family of closed embedded
convex plane curves contract to a circular point as $t$ approaches the final
time $T$. As seen in section~\ref{sec:C-E estimate}, the curvature $k$ tends to infinity
as $t\to T$ because the area vanishes. In order to control the blow up of
curvature, we are going to rescale the curve in such a way that its enclosed
area approaches a positive constant. The proof relies on a uniform bound on the
rescaled curvature.

Let us introduce parametrisation by angle.
We define the normal angle $\vartheta$ to be the angle made by the inner normal
of a curve $\gamma$ and the positive x-axis.
For a fixed time, the simple, closed, convex curve $\gamma$ can be parametrised
by normal angle $\vartheta$ in such a way that $\gamma =\gamma(\vartheta)$,
with inner unit normal $\nu(\vartheta )= -(\cos\vartheta , \sin\vartheta)$.
Under such a parametrisation, the support function of $\gamma$, denoted
$h(\gamma)=\IP{\gamma}{-\nu}$, takes the special form $h(\gamma(\vartheta))=
\IP{\gamma(\vartheta)}{(\cos\vartheta , \sin\vartheta)}$. 

We recall two results regarding this parametrisation that are independent of the flow.

\begin{lem}\label{LemSupportFnAndCurvature}
Let $\gamma:\S^1\to\R$ be a closed embedded convex plane curve. The support function $h:\S^1\to\R$ defined as $h(\gamma(\vartheta))=\IP{\gamma}{-\nu}= \IP{\gamma(\vartheta)}{(\cos\vartheta , \sin\vartheta)}$ is related to the curvature $k(\vartheta)$ in such way
\[
h+ h_{\vartheta \vartheta} = \frac{1}{k}.
\]
\end{lem}
Let us define the \emph{width} of a convex curve in the $(\cos\vartheta , \sin\vartheta)$ direction to be 
\[
w(\vartheta)=h(\vartheta)+h(\vartheta+\pi).
\]
The next lemma relates the width to the entropy 
\[
E(\gamma(\vartheta))=\frac{1}{|\gamma|}\int_{\gamma} \log(k)\,d\vartheta
\,.
\]
\begin{lem}\label{LemSupportFnAndWidth}
Let $\gamma:\S^1\to\R$ be a closed, embedded, convex plane curve with support function $h:\S^1\to\R$ as defined in Lemma~\ref{LemSupportFnAndCurvature}. There exists a positive constant $C$ such that for all $\vartheta\in \S^1$,
	$$w(\vartheta)\geq Ce^{-E(\gamma)}.$$
	Furthermore, if the entropy $E(\gamma)$ is uniformly bounded from above, then $w(\vartheta)>0$.
\end{lem}

Next, let us derive an a priori estimate on the speed $F(\theta,t)$ of the flow~\eqref{C-E problem}.

\begin{prop}\label{FmaxGrdientBound}
Let $\gamma:\S^1\times[0,T)\to\R^2$ be a family of closed, embedded, convex
plane curves evolving by the flow~\eqref{C-E problem}.
Define
\begin{equation}
\label{Mdef}
M^2={\sup_{\theta\in[0,2\pi)}{(F^2+F_\theta^2)(\theta,0)}}\,.
\end{equation}
The speed of the flow $F(\theta,t)$ satisfies:
\begin{equation}
\label{Fthetabound}
\sup_{\theta\in[0,2\pi)} |F_\theta(\theta,t)| \le M + \int_0^{2\pi} |F(\theta,t)|\,d\theta\,,\text{ and}
\end{equation}
\begin{equation}\label{FmaxBound}
F_{\max}(t) \leq M_1\left(1+\int_0^{2\pi}|F(\theta,t)|\,d\theta  \right)\,,
\end{equation}
where $M_1=\max\{2\pi M, 2\pi+(2\pi)^{-1} \}$, and
\begin{equation}\label{FmaxPartialBound}
F_{\max}(t) \leq 2 F(\theta,t) + \frac{M}{2\pi}
\end{equation}
for all $\theta$ satisfying $|\theta-\theta^*(t)|\leq \frac{1}{4\pi}$, where
$\theta^*$ is such that $F(\theta,t) \le F(\theta^*(t),t)$ for all $\theta\in[0,2\pi]$.
\end{prop}

The proof uses the following lemmata:

\begin{lem}\label{Fevolution}
The speed of the flow $F:\S^1\times[0,T)\to\R$ satisfies
\begin{equation}
F_t = \sigma_1 k^2(F_{\theta\theta} + F).
\end{equation}
\end{lem}
\begin{proof}
This follows by combining \eqref{C-E kPDE angle} with the definition of $F$ (recall \eqref{C-E problem}).
\end{proof}

\begin{lem}\label{Ftimederivative}
Let $F:\S^1\times[0,T)\to\R$ be the speed of the flow~\eqref{C-E problem}.
Set $M^2$ as in \eqref{Mdef}.
Then $(F^2+F_\theta^2)(\theta,t) > M^2$ implies $F_t(\theta,t) \geq 0$.
\end{lem}
\begin{proof}
Fix a point $(\theta_0, t_0)$ for some $t_0>0$, and let 
\begin{equation}
B=(F^2+F_\theta^2)^{1/2}(\theta_0,t_0).
\end{equation} 
Assume that $B>M$.
We need to show that $F_t(\theta_0,t_0)\geq 0$, which is equivalent to showing
$(F_{\theta\theta} + F)(\theta_0,t_0)\geq 0$ (by Lemma \ref{Fevolution}).

Pick $\xi\in(-\pi,\pi)$ such that $F(\theta_0,t_0)=B\cos\xi$ and $F_\theta(\theta_0,t_0)=-B\sin\xi$.
This is possible due to the differentiability of $F$.
Let us denote 
\[F^*(\theta)=B\cos(\theta-\theta_0+\xi),\] 
and consider the function 
\[G(\theta, t)=F(\theta,t)-F^*(\theta).\] 
It is straightforward to see that $G(\theta_0,t_0)=0$ and
$G_\theta(\theta_0,t_0)=0$, hence the point $(\theta_0,t_0)$ is a double root
for $G(\cdot,t_0)$.

Now we aim to show that, using the Sturmian oscillation theorem, this must be the only root for $G$. 
Since $F(\theta,t)$, and hence $G(\theta,t)$, are $2\pi$-periodic functions,
without loss of generality, we perform analysis on the interval
$(\theta_0-\xi-\pi,\theta_0-\xi+\pi)$. As $B>M$ and $(F^2+F_\theta^2)$ is initially
bounded by $M^2$, we obtain
$$G(\theta_0-\xi,0)=F(\theta_0-\xi,0)-B \leq \sup|F(\theta,0)|-B< 0.$$
So on at least one interval, $G$ is initially negative.
At the end points $\theta=\theta_0-\xi\pm\pi$, we have
\begin{align*}
G(\theta_0-\xi\pm\pi,0)&=F(\theta_0-\xi\pm\pi,0)-F^*(\theta_0-\xi\pm\pi,0)=F(\theta_0-\xi\pm\pi,0)-B\cos(\pm\pi) \\
&\geq F(\theta_0-\xi-\pi,0) + B > 0\,.
\end{align*}
Therefore, there is at least two open interval on which $G$ is initially positive.
This implies that there are at least two roots for $G$ at time $t=0$.
Consider $\theta_1$ in the interval $(\theta_0-\xi-\pi,\theta_0-\xi)$.
We have $\sin(\theta_1-\theta_0+\xi) <0$.
If $\theta_1$ is an initial root of $G$, (that is, $G(\theta_1,0) = 0$) then $F(\theta_1,0)=F^*(\theta_1)$.
We compute
\begin{align*}
G_\theta(\theta_1,0) &= F_\theta(\theta_1,0) + B\sin(\theta_1-\theta_0+\xi) \\
&< \left(B^2-F^2(\theta_1,0)\right)^{1/2}+ B\sin(\theta_1-\theta_0+\xi)
\\
&= \left(B^2-B^2\cos^2(\theta_1-\theta_0+\xi)\right)^{1/2}+ B\sin(\theta_1-\theta_0+\xi)
 = 0\,.
\end{align*}
Thus $G$ is decreasing at every root on the interval $(\theta_0-\xi-\pi,\theta_0-\xi)$.
This means that there can be only one simple root in the interval $(\theta_0-\xi-\pi,\theta_0-\xi)$.
Indeed, if there were another root for $G$ after the first, then at this root
we must have $G_\theta \ge 0$, which is impossible by the above calculation.
A similar argument applies on $(\theta_0-\xi,\theta_0-\xi+\pi)$, yielding that at any root $\theta_2$ in this interval we have $G_\theta(\theta_2,0)>0$.
We conclude that $G$ has only two simple roots at $t=0$. 
 
By Sturmian theory (see for example \cite{galaktionov2004geometric}), for
$G(\theta,t)$ satisfying the parabolic differential equation
$G_t=\sigma_1k^2G_{\theta\theta}$, with two simple roots at initial time over
its entire domain, there can be at most two simple roots (counted with
multiplicity) at all later times $t=t_0>0$.
As $G(\theta_0,t_0)$ is a double root, it must be the only root for $G(\cdot,t_0)$. 

To show that $G(\theta_0,t_0)$ is a minimum value for $G(\cdot,t_0)$, we consider
\begin{align*}
G\left(\theta_0-\xi\pm\frac{\pi}{2},t_0\right) &= F\left(\theta_0-\xi\pm\frac{\pi}{2},t_0\right)-B\cos\left(\pm\frac{\pi}{2}\right) \\
&= F\left(\theta_0-\xi\pm\frac{\pi}{2},t_0\right)\,.
\end{align*}
Since $F$ is initially positive by convexity and the assumption that $\sigma_i \ge 0$, the maximum principle implies that $F$ remains positive.
Therefore 
$G(\theta_0-\xi\pm\frac{\pi}{2},t_0) > 0$.
As we have deduced from the Sturmian theorem that $(\theta_0,t_0)$ is the only
zero for $G(\cdot,t_0)$, $G(\cdot,t_0)$ is non-negative and so
$G(\theta_0,t_0)$ is a strict local minimum value for $G(\cdot,t_0)$, as
required.
Therefore $G_{\theta\theta}(\theta_0,t_0)=(F_{\theta\theta}+F)(\theta_0,t_0)$ is non-negative.
We have proved that
\[
(F^2+F_\theta^2)^\frac12(\theta_0,t_0) > M\quad\Longrightarrow\quad (F_{\theta\theta}+F)(\theta_0,t_0) \ge 0
\]
which implies $F_t(\theta_0,t_0) \ge 0$, as required.
\end{proof}

We are now ready to prove Proposition \ref{FmaxGrdientBound}.

\begin{proof}[Proof of Proposition~\ref{FmaxGrdientBound}]
Let us fix a time $t_1>0$ and let $M^2$ be as in \eqref{Mdef}.
Since estimate \eqref{Fthetabound} is trivial for $t=0$, we may take
\[
t_0 = \sup\{t\in[0,T)\,:\,\eqref{Fthetabound}\text{ holds}\}\,.
\]
Since $\sigma_1, \sigma_2$ and $k$ are non-negative, $F$ may not vanish completely at any given time.
Therefore by continuity of the flow $t_0 > 0$.

Suppose that $t_0 < T$.
Let $t_1 = t_0+\varepsilon$ where $\varepsilon\in(0,T-t_0)$ is a small parameter.
If $(F_\theta^2+F^2)(\theta,t_1) \leq M^2$, then obviously $|F_\theta
(\theta,t_1)| \leq M$ and $F_{\max}(t_1)$ is also bounded by $M$.
This is (stronger than) the estimate in \eqref{Fthetabound}.
This contradicts the definition of $t_0$.

Therefore it must be the case that $(F_\theta^2+F^2)(\theta,t_1) > M^2$ for some $\theta\in[0,2\pi)$.
This implies that there exists a first $\theta_1\in[0,2\pi)$ such that
$(F_\theta^2+F^2)(\theta_1,t_1) = M^2$ and $(F_\theta^2+F^2)(\theta,t_1) < M^2$
for $\theta\in[0,\theta_1)$.
There exists a maximal interval $(\theta_1,\theta_2)$ such that
$(F_\theta^2+F^2)(\theta,t_1) > M^2$ for $\theta\in(\theta_1,\theta_2)$.
Then by Lemma~\ref{Ftimederivative}
\begin{equation}
\label{EQaux1}
F_t(\theta,t_1)\geq 0\quad
\text{for all $\theta\in(\theta_1,\theta_2)$.}
\end{equation}
Note that by continuity of $F_\theta$, for $\varepsilon$ sufficiently small it can not be the case that $\theta_1 = 0$ and $\theta_2 = 2\pi$.
Indeed, we have that $|F_\theta(\theta,t_0)| \le M + \int_0^{2\pi}|F(\theta,t_0) |\,d\theta$ with equality achieved at one or more isolated points.

Note that $F_\theta$ may not achieve equality in this estimate on any open interval, otherwise by analyticity of the flow it is constant on $[0,2\pi)$ and then $F_{\theta\theta}(\theta,t_0) = 0$, which means $k_{\theta}$ is constant.
As the flow is closed, $k$ must be itself a constant, so that $\gamma(\cdot,t_0)$ is a circle.
But then $F_\theta = 0$, so $F_\theta$ did not achieve equality after all.
In fact, $\theta_2(\varepsilon) \searrow \theta_1(\varepsilon)$ as
$\varepsilon\searrow0$, so clearly we may assume that $\theta_2 - \theta_1 <
2\pi$ by taking $\varepsilon$ sufficiently small.

This means that $(F_\theta^2+F^2)(\theta_2,t_1) = M^2$, in particular that
$F_\theta^2(\theta_2,t_1) \le M^2$.
Using Lemma \ref{Ftimederivative} we find
\begin{align*}
&\int_{\theta}^{\theta_2} (F_{\theta\theta}+F)(\theta,t_1) \,d\theta \geq 0\,,
\end{align*}
where $\theta\in(\theta_1,\theta_2)$.
Rearranging yields
\begin{align*}
|F_{\theta}(\theta,t_1)| &\leq \left|F_{\theta}(\theta_2,t_1) +\int_{\theta}^{\theta_2} F(\theta,t_1) \,d\theta\right|  \\
&\leq M + \int_0^{2\pi} |F(\theta,t_1)| \,d\theta\,.
\end{align*}
Now by assumption we have $F_\theta^2(\theta,t_1) \le M^2$ for all $\theta\in[0,\theta_1]$.
Combining this with the above we have
\begin{align*}
\sup_{\theta\in[0,\theta_2]} |F_\theta (\theta,t_1)| \leq  M + \int_0^{\theta_2} |F(\theta,t_1)| \,d\theta.
\end{align*}
Finally, the above argument holds on all maximal intervals $(\theta_1,\theta_2)$ where the estimate $F^2 + F_\theta^2 \le M^2$ fails.
Arguing as above for each of these intervals, we obtain the estimate
\begin{align*}
\sup_{\theta\in[0,2\pi]} |F_\theta (\theta,t_1)| \leq  M + \int_0^{2\pi} |F(\theta,t_1)| \,d\theta.
\end{align*}
This implies that \eqref{Fthetabound} holds for $t=t_1>t_0$, which is a contradiction.
Therefore $t_0 = T$.

Now we show the remaining estimates in Proposition \ref{FmaxGrdientBound}.
By fundamental theorem of calculus,
\begin{align*}
	F_{\max}(t) &\leq \frac1{2\pi}\left|\int_0^{2\pi}F(\theta,t)\,d\theta\right|
	                  + \int_0^{2\pi} |F_\theta (\theta,t)|\,d\theta
\\
&\leq \frac{1}{2\pi}\int_0^{2\pi} |F(\theta,t)|\,d\theta + 2\pi \left(M + \int_0^{2\pi} |F(\theta,t)| \,d\theta \right) \\
&\leq M_1\left(1 + \int_0^{2\pi} |F(\theta,t)| \,d\theta \right)
\end{align*}
where $M_1=\max\{2\pi M, 2\pi+(2\pi)^{-1} \}$.

Let $\theta^*$ be such that $F_{\max}(t)=F(\theta^*,t)$.
Let $\theta\in(0,2\pi)$ be such that $|\theta-\theta^*|\leq
\frac{1}{4\pi}$. The fundamential theorem of calculus implies
\begin{align*}
F_{\max}(t) &= F(\theta,t) +\int_{\theta}^{\theta^*} F_\theta (\theta,t)\,d\theta\\
&\leq  F(\theta,t) + |\theta^*-\theta| \sup_{\theta\in[0,2\pi]}|F_\theta (\theta,t)| \\
&\leq  F(\theta,t) + |\theta^*-\theta| \left(M + \int_0^{2\pi} |F(\theta,t)| \,d\theta \right) \\
&\leq F(\theta,t) + \frac{1}{4\pi}\left(M+2\pi F_{\max}(t)\right).
\end{align*}
Making $F_{\max}(t)$ the subject, we obtain
\begin{align*}
F_{\max}(t) 
&\leq 2F(\theta,t)+\frac{M}{2\pi}
\end{align*}
as desired.
\end{proof}


\section{Curvature estimate for the continuous rescaling}

In this section we carry through details of the calculations from Chou-Zhu in
the special case of our family of flows here, for the convenience of the
reader.
Briefly, once we have a curvature bound in hand, we will obtain convergence of
the rescaled flow to a limit (using the same argument as Huisken
\cite{huisken1990asymptotic}), and then this limit must satisfy $Q=0$ by
Theorem \ref{thmmonot}.

In order to obtain pointwise information on the blowup of the curvature, we use the assumption $\sigma_1\ge1$ together with the following argument.

Our goal is to prove an a-priori estimate for the rescaled curvature $\hat k$.
Although the above lemma gives useful information on the blowup rate for the $L^2$-norm of curvature, it does not give any information at all on the pointwise blowup of $k$.
In the next lemma, we attack this problem, first showing that its blowup rate is at worst subexponential in rescaled time.

\begin{lem}\label{kmaxbound}
Let $\hat{\gamma}:\S^1\times[0,\infty)\to \R^2$ be a solution to the rescaled
flow \eqref{mainFlowRescaled} with convex initial data.
The maximum of rescaled curvature $\hat{k}_{\max}(\hat{t})$ satisfies
\[
\lim_{\hat{t}\to \infty} e^{-\hat{t}}\hat{k}_{\max}(\hat{t}) =0.
\]
\end{lem}
\begin{proof}
Note first that
\[
t_{\hat t} = 2Te^{-2\hat t} = 2T - 2T(1 - e^{-2\hat t})
 = (2T - 2t) = \phi^{-2}
\,.
\]
We now differentiate length $L(t(\hat{t}))$ with respect to the normalised
time variable $\hat{t}$.
Using Proposition \ref{lenghtevolution} and the above, we compute
\begin{align*}
L_{\hat{t}} &= L_t t_{\hat{t}} = \left(-\sigma_1\int_\gamma k^2 \,ds - 2\pi\sigma_2\right)\phi^{-2} \\
&= -2Te^{-2\hat{t}}\sigma_1\int_{\S^1} k \,d\theta -4Te^{-2\hat{t}}\pi\sigma_2\,.
\numberthis{\label{C-E Zhu rescaled length evolution}}
\end{align*}
By Theorem \ref{LAasympt}, the length vanishes at final time $t=T$, and so we have $\lim_{\hat{t}\to\infty}L(t(\hat{t}))=0$.
Integrating~\eqref{C-E Zhu rescaled length evolution} with respect to $\hat{t}$ yields
\begin{align*}
-L(0) = \int_0^\infty L_{\hat{t}}\,d\hat{t} = -2T\int_0^\infty\int_0^{2\pi} \sigma_1{k}e^{-2\hat{t}} \,d\theta d\hat{t}-4\pi T\int_0^\infty \sigma_2e^{-2\hat{t}} \,d\hat{t}\,.
\end{align*}
That is,
\begin{equation}\label{kL1InTimeBoud}
\int_0^\infty\int_0^{2\pi} Fe^{-2\hat{t}} \,d\theta d\hat{t}
= \frac{L(0)}{2T}
\,.
\end{equation}
Integrating the estimate \eqref{FmaxBound}, we find
\begin{align*}
\int_0^\infty F_{\max}e^{-2\hat{t}}\,d\hat{t} &\leq  \int_0^\infty e^{-2\hat{t}}M_1\left(1+\int_0^{2\pi}|F|\,d\theta  \right)\,d\hat{t} 
\end{align*}
where $M_1$ is as in Proposition \ref{FmaxGrdientBound}.
Convexity implies $F \ge 0$ and so, using \eqref{kL1InTimeBoud}, we see that
\begin{equation}\label{Fmax L1 in time}
\begin{aligned}
\int_0^\infty F_{\max}\left(\hat{t}\right)e^{-2\hat{t}}\,d\hat{t}
&\le \frac{M_1}{2} + \frac{L(0)}{2T}\,.
\end{aligned}
\end{equation}
We know that the curvature $k(t(\hat{t}))$ blows up as $\hat{t}\longrightarrow\infty$.
This implies $F(\theta,\hat{t})=\sigma_1 k+\sigma_2$ can not have a uniform upper bound.
In particular, it is eventually strictly larger than the constant $M^2$ (see \eqref{Mdef}).
Call this time $\hat t_*$.
Lemma \ref{Ftimederivative} implies that $F_{\hat{t}}\left(\theta,\hat{t}\right)\geq 0$ on an open interval $[\hat t_*,\hat t_* + \varepsilon)$.
Now, as $F$ is non-decreasing along this interval, we see that $F$ remains strictly larger than $M^2$ and indeed may never decrease again.
Taking $\varepsilon$ to be maximal, we find that $F_{\hat t} \ge 0$ on $[\hat t_*,\infty)$.

Since $\sigma_1 > 0$, this means that $k$ is monotone increasing for all $\hat t > \hat t_*$.
In particular the maximum of $F$ is monotone increasing.
The $L^1$ in rescaled time bound \eqref{Fmax L1 in time}, yields
\[
(F_{\max}e^{-2\hat t})(\hat t_j) \rightarrow 0
\]
along a subsequence $\{\hat t_j\}_{j=1}^\infty$, $\hat t_j \rightarrow \infty$.
Since $F$ is monotone, the limit is independent of subsequence and we conclude
\[
0=\lim_{\hat{t}\to\infty}F_{\max}(\hat{t})e^{-2\hat{t}}
 = \lim_{\hat{t}\to\infty}(\sigma_1 k_{\max}(\hat{t})+\sigma_2)e^{-2\hat{t}} = \lim_{\hat{t}\to\infty} \sigma_1\frac{e^{-\hat{t} }}{\sqrt{2T}}\hat{k}_{\max}(\hat{t})
\,,
\]
as required.
\end{proof}

Define
\begin{equation}\label{entropy define u}
u(\hat{\theta},\hat{t}) := -1 +\sigma_1\hat{k}(\hat{k}_{\hat{\theta}\hat{\theta}} +\hat{k}) -2 \sqrt{2T}e^{-\hat{t}}\sigma_2\hat{k}.
\end{equation}
A key step in obtaining a uniform bound for the entropy of the rescaled flow is the following monotonicity of the integral of $u$.

\begin{prop}\label{prop entropy u negative}
Let $\hat{\gamma}:\S^1\times[0,\infty)\to \R^2$ be a solution to the rescaled
flow \eqref{mainFlowRescaled} with convex initial data.
There exists a positive constant $\hat{t}_0$ such that 
\[
f(\hat{t}) := \int_0^{2\pi} u\,d\hat{\theta}
\]
is non-positive for all $\hat{t}>\hat{t}_0$.
\end{prop}
\begin{proof}
First, we calculate
\begin{equation}\label{ktPDEangle2}
\begin{aligned}
\hat{k}_{\hat{t}} &= \left(\frac{k}{\phi}\right)_{\hat{t}} = \left(\pD{(\phi^{-1})}{t}k+\phi^{-1}\pD{k}{t}\right)\pD{t}{\hat{t}} \\
&= -\hat{k} + \sigma_1\hat{k}^2\left(\hat{k}_{\hat{\theta}\hat{\theta}}+\hat{k}\right) + \sqrt{2T}e^{-\hat{t}}\sigma_2\hat{k}^2.
\end{aligned}
\end{equation}
Using \eqref{ktPDEangle2}, we compute the evolution of the entropy
\begin{equation}\label{entropy Rescaled}
\hat E(\hat{t}) := E(\hat{\gamma}(\hat{\theta} ,\hat{t})) = \frac{1}{2\pi}\int_0^{2\pi} \log\hat{k}\,d\hat{\theta}
\end{equation}
in terms of $f$ and an error term:
\begin{align*}
2\pi\hat E_{\hat{t}}
 &= \int_0^{2\pi}\frac{\hat{k}_{\hat{t} }}{\hat{k}} d\hat{\theta} = \int_0^{2\pi}-1+\sigma_1\hat{k}\left(\hat{k}_{\hat{\theta}\hat{\theta}}+\hat{k}\right) + \sqrt{2T}e^{-\hat{t}}\sigma_2\hat{k} \,d\hat{\theta}
\\
&= \int_0^{2\pi}-1 +\sigma_1\hat{k}\left(\hat{k}_{\hat{\theta}\hat{\theta}} +\hat{k}\right) -2 \sqrt{2T}e^{-\hat{t}}\sigma_2\hat{k} \,d\hat{\theta} + \int_0^{2\pi} 3\sqrt{2T}e^{-\hat{t}}\sigma_2\hat{k} \,d\hat{\theta}
\\
&= \int_0^{2\pi} u \,d\hat{\theta} + \int_0^{2\pi} 3\sqrt{2T}e^{-\hat{t}}\sigma_2\hat{k} \,d\hat{\theta}
\\
&= f + \int_0^{2\pi} 3\sqrt{2T}e^{-\hat{t}}\sigma_2\hat{k} \,d\hat{\theta}\,.
\numberthis{\label{entropy evolution1}}
\end{align*}
We shall compute the evolution of $f(\hat{t})$. We first note the following equality, that follows easily from integration by parts:
\begin{align*}
f(\hat{t})=\int_0^{2\pi}-1 -\sigma_1\hat{k}_{\hat{\theta}}^2 + \sigma_1\hat{k}^2 -2 \sqrt{2T}e^{-\hat{t}}\sigma_2\hat{k} \,d\hat{\theta}.
\end{align*}
Recall that the angle paramater $\hat{\theta}$ is chosen to commute with the time parameter $\hat{t}$. We use this and integration by parts to compute the evolution of $f(\hat{t})$:
\begin{align*}
f'(\hat{t})
&=\int_0^{2\pi} -2\sigma_1\hat{k}_{\hat{\theta}}\hat{k}_{\hat{\theta}\hat{t}}+2\sigma_1\hat{k}\hat{k}_{\hat{t}}
 + 2\sqrt{2T}e^{-\hat{t}}\sigma_2\hat{k} -2\sqrt{2T}e^{-\hat{t}}\sigma_2\hat{k}_{\hat{t}} \,d\hat{\theta} \\
&= 2\int_0^{2\pi} \left[\sigma_1\left(\hat{k}_{\hat{\theta}\hat{\theta}}+\hat{k}\right)-\sqrt{2T}e^{-\hat{t}}\sigma_2\right]\hat{k}_{\hat{t}}
+ \sqrt{2T}e^{-\hat{t}}\sigma_2\hat{k}\,d\hat{\theta}.
\end{align*}
Expanding out the expression of $\hat{k}_{\hat{t}}$ using~\eqref{ktPDEangle2}, we find 
\begin{align*}
f'(\hat{t})
= 2\int_0^{2\pi} &\left[\sigma_1\left(\hat{k}_{\hat{\theta}\hat{\theta}}+\hat{k}\right)-\sqrt{2T}e^{-\hat{t}}\sigma_2\right]
\\
                  &\left[-\hat{k} + \sigma_1\hat{k}^2\left(\hat{k}_{\hat{\theta}\hat{\theta}}+\hat{k}\right) + \sqrt{2T}e^{-\hat{t}}\sigma_2\hat{k}^2\right]\,d\hat{\theta}\\
&\hskip-1cm + 2\int_0^{2\pi} \sqrt{2T}e^{-\hat{t}}\sigma_2\hat{k}\,d\hat{\theta}
\\
&\hskip-1.5cm = 2\int_0^{2\pi} -\sigma_1\hat{k}\left(\hat{k}_{\hat{\theta}\hat{\theta}}+\hat{k}\right) +\sigma_1^2\hat{k}^2\left(\hat{k}_{\hat{\theta}\hat{\theta}}+\hat{k}\right)^2
\\
                              &+ 2\sqrt{2T}e^{-\hat{t}}\sigma_2\hat{k}
                               - 2Te^{-2\hat{t}}\sigma_2^2\hat{k}^2
                               \,.
\numberthis{\label{entropy cal step1}}
\end{align*}
We wish to rewrite the above in terms of $u$ and $u^2$. We note that
\begin{align*}
u^2 = 1 &+\sigma_1^2\hat{k}^2\left(\hat{k}_{\hat{\theta}\hat{\theta}} +\hat{k}\right)^2 + 8Te^{-2\hat{t}}\sigma_2^2\hat{k}^2
\\& -2\sigma_1\hat{k}\left(\hat{k}_{\hat{\theta}\hat{\theta}}+\hat{k}\right)+4\sqrt{2T}e^{-\hat{t}}\sigma_2\hat{k}
    -4\sqrt{2T}e^{-\hat{t}}\sigma_2\sigma_1\hat{k}^2\left(\hat{k}_{\hat{\theta}\hat{\theta}}+\hat{k}\right).
\end{align*}
Comparing the integrand in~\eqref{entropy cal step1} with $u$ and $u^2$, we obtain
\begin{align*}
f'(\hat{t})
&= 2\int_0^{2\pi} u^2 \,d\hat{\theta} + 2\int_0^{2\pi} -1 - 2\sqrt{2T}e^{-\hat{t}}\sigma_2\hat{k} 
 +\sigma_1\hat{k}\left(\hat{k}_{\hat{\theta}\hat{\theta}}+\hat{k}\right) \\
&\qquad 
- 10Te^{-2\hat{t}}\sigma_2^2\hat{k}^2
                               +4\sqrt{2T}e^{-\hat{t}}\sigma_2\sigma_1\hat{k}^2\left(\hat{k}_{\hat{\theta}\hat{\theta}}+\hat{k}\right)
 \,d\hat{\theta} \\
&= 2\int_0^{2\pi} u^2 \,d\hat{\theta} + 2\int_0^{2\pi} u\,d\hat{\theta}\\
&\qquad + 2\int_0^{2\pi} 
- 10Te^{-2\hat{t}}\sigma_2^2\hat{k}^2                             
    +4\sqrt{2T}e^{-\hat{t}}\sigma_2\sigma_1\hat{k}^2\left(\hat{k}_{\hat{\theta}\hat{\theta}}+\hat{k}\right)
 \,d\hat{\theta}
\,.\numberthis{\label{entropy cal step2}}
\end{align*}
We wish to absorb the last term of \eqref{entropy cal step2} into the other terms.
First, we calculate
\begin{align*}
	2\int_0^{2\pi} 4&\sqrt{2T}e^{-\hat{t}}\sigma_2\sigma_1\hat{k}^2\left(\hat{k}_{\hat{\theta}\hat{\theta}}+\hat{k}\right)  
 \,d\hat{\theta} 
\\&= \int_0^{2\pi} 8\sqrt{2T}\sigma_2e^{-\hat{t}}\left(\hat{k}u +\hat{k} + 2\sqrt{2T}\sigma_2e^{-\hat{t}}\hat{k}^2\right)\,d\hat{\theta}
\\
&= \int_0^{2\pi} 8\sqrt{2T}\sigma_2e^{-\hat{t}}\hat{k}u \,d\hat{\theta}
+ \int_0^{2\pi} 8\sqrt{2T}\sigma_2e^{-\hat{t}}\hat{k}\,d\hat{\theta}
\\&\qquad
+ \int_0^{2\pi} 32T\sigma_2^2e^{-2\hat{t}}\hat{k}^2\,d\hat{\theta}
\,.
\end{align*}
We then apply the inequality $ab \ge -\frac14 a^2 - b^2$ to the RHS above, yielding
\begin{align*}
	2\int_0^{2\pi} 4&\sqrt{2T}e^{-\hat{t}}\sigma_2\sigma_1\hat{k}^2\left(\hat{k}_{\hat{\theta}\hat{\theta}}+\hat{k}\right)  
 \,d\hat{\theta} 
 \\
&\geq -\int_0^{2\pi}\frac{1}{4}64(2T)\sigma_2^2e^{-2\hat{t}}\hat{k}^2\,d\hat{\theta} -\int_0^{2\pi}u^2\,d\hat{\theta}
+ \int_0^{2\pi} 8\sqrt{2T}\sigma_2e^{-\hat{t}}\hat{k}\,d\hat{\theta}
\\&\qquad
+ \int_0^{2\pi} 32T\sigma_2^2e^{-2\hat{t}}\hat{k}^2\,d\hat{\theta}
\\
&= - \int_0^{2\pi} u^2 \,d\hat{\theta}
+ \int_0^{2\pi} 8\sqrt{2T}\sigma_2e^{-\hat{t}}\hat{k}\,d\hat{\theta}
\,.
\end{align*}
Combining with \eqref{entropy cal step2}, we have
\begin{align}\label{entropy cal step25}
f'(\hat{t})
&\geq \int_0^{2\pi} u^2 \,d\hat{\theta} + 2\int_0^{2\pi} u\,d\hat{\theta}
 - 20T\sigma_2^2\int_0^{2\pi} e^{-2\hat{t}}\hat{k}^2 \,d\hat{\theta}
\\&\qquad
 + 2\sqrt{2T}\sigma_2\int_0^{2\pi} 4e^{-\hat{t}}\hat{k}\,d\hat{\theta}
\notag\,.
\end{align}
It follows from Lemma \ref{kmaxbound} that for every $\delta>0$, there exists a $\hat{t}_\delta$ such that for every $\hat{t}>\hat{t}_\delta$,
\begin{equation}\label{entropy cal step4}
	e^{-\hat{t}}\hat{k}_{\max}\left(\hat{t}\right) \leq \delta.
\end{equation}
Let us briefly analyse the latter two integrals on the RHS of \eqref{entropy cal step25}.
For $\hat t > \hat{t}_\delta$, we have
\begin{align*}
	-  20T\sigma_2^2\int_0^{2\pi} &e^{-2\hat{t}}\hat{k}^2 \,d\hat{\theta}
 + 2\sqrt{2T}\sigma_2\int_0^{2\pi} 4e^{-\hat{t}}\hat{k}\,d\hat{\theta}
\\
&= 2\sqrt T\sigma_2\int_0^{2\pi} e^{-\hat t}\hat k\big(\sqrt{2}(4) - 10\sqrt{T}\sigma_2e^{-\hat t}\hat k\big)\,d\hat\theta
\\
&\ge0
\end{align*}
if
\[
\sqrt{2}(4) - 10\sqrt{T}\sigma_2e^{-\hat t}\hat k
	\ge 0\,,
\]
which is satisfied so long as $e^{-\hat t}\hat k \le \frac{\sqrt{2}(4)}{10\sqrt{T}\sigma_2}$.
Therefore, if we choose $\delta = \frac{\sqrt{2}(4)}{10\sqrt{T}\sigma_2}$, by \eqref{entropy cal step4}
we have for $\hat t > \hat t_\delta$
\begin{align}\label{entropy cal step3}
f'(\hat{t})
&\geq \int_0^{2\pi} u^2 \,d\hat{\theta} + 2\int_0^{2\pi} u\,d\hat{\theta}
\,.
\end{align}
Using H\"older's inequality,
we conclude that
for all $\hat{t}>\hat{t}_\delta$ (with $\delta$ chosen as above),
\begin{align*}
f'(\hat{t})
 &\geq \int_0^{2\pi} u^2 \,d\hat{\theta} +2\int_0^{2\pi} u\,d\hat{\theta}
\\
&\geq \frac{1}{2\pi}\left( \int_0^{2\pi} u\,d\hat{\theta}\right)^2 +2\int_0^{2\pi} u\,d\hat{\theta}
\\
&\geq \frac{1}{2\pi}f^2(\hat{t})+2f(\hat{t})\numberthis{\label{entropy cal step5}}
\,.
\end{align*}

We are now ready to finish the proof.
Our proof is by contradiction.
Fix a positive constant $\delta$.
We denote $\hat{t}_\delta$ to be the value such that \eqref{entropy cal step5} is valid for all $\hat{t}>\hat{t}_\delta$.
Let $c>0$ be arbitrary.
Suppose that there exists $\hat{t}_1>\hat{t}_\delta$ such that $f(\hat{t}_1)\geq c>0$. Thus, we have
from~\eqref{entropy cal step5} that
\begin{equation}\label{entropy cal step6}
f'(\hat{t}_1)\geq \frac{1}{2\pi}f^2(\hat{t}_1)+2f(\hat{t}_1)
 >0
\,.
\end{equation}
This implies
\[
f(\hat{t})>c>0, \text{ for all }  \hat{t} >\hat{t}_1.
\]
Hence the inequality~\eqref{entropy cal step6} holds for all $\hat{t}>\hat{t}_1$. We can rearrange~\eqref{entropy cal step6} and integrate over time to see
\[
	\int_{\hat{t}_1}^{\hat{t}} \frac{f'(\hat{t})}{f^2(\hat{t})} \,d\hat{t} \geq \int_{\hat{t}_1}^{\hat{t}} \frac1{2\pi} \,d\hat{t}
\,,
\]
or
$$-\frac{1}{f(\hat{t})} \geq \frac1{2\pi}(\hat{t}-\hat{t}_1)-\frac{1}{f(\hat{t}_1)}.$$
The positivity of $f(\hat{t})$ implies
\begin{align*}
	f(\hat{t}) \geq \frac{f(\hat{t}_1)}{1-f(\hat{t}_1)\frac1{2\pi}(\hat{t}-\hat{t}_1)}.
\end{align*}
Let us take the sequence $\{\hat{t}_i\}$ such that $\hat{t}_i\longrightarrow \hat{t}^*$, where $\hat{t}^*= \hat{t}_1 +\frac{1}{f(\hat{t}_1)\frac1{2\pi}}$ is a finite number.
It is clear that $f(\hat{t}_i)$ blows up as $\hat{t}_i$ tends to $\hat{t}^*$.
But, $\hat t^*$ is a finite time and as the flow is smooth, the quantity $f(\hat t^*)$ is bounded.
This is a contradiction.

Therefore, for every $c > 0$,
\[
f(\hat t) < c
\quad\text{ for all }\quad \hat t > \hat t_\delta
\,.
\]
But this can only be the case if $f(\hat t) \le 0$ for all $\hat t > \hat t_\delta$.
This finishes the proof.
\end{proof}

Next we show that the rescaled entropy $\hat E(\hat{t})$ is uniformly bounded for all $\hat{t}\in[0,\infty)$. 

\begin{prop}
\label{PRentropybddrescaled}
Let $\hat{\gamma}:\S^1\times[0,\infty)\to \R^2$ be a solution to the rescaled
flow \eqref{mainFlowRescaled} with convex initial data.
Then for all $\hat t \in [0,\infty)$ the rescaled entropy satisfies
\begin{equation}\label{EntBound}
	\hat E(\hat{t}) \leq C\,,
\end{equation}
where
\[
	C = \max\Big\{
		\sup_{\hat t \in [0,\hat t_0)}\{\hat E(\hat t)\},
        	\hat E(\hat{t}_0)
 		+ \frac{3\sigma_2 L(0)}{2\pi\sigma_1}
		- \frac{3T\sigma_2^2}{\sigma_1}
		\Big\}
		\,.
\]
Here $\hat t_0$ is as in the statement of Proposition \ref{prop entropy u
negative}.
\end{prop}
\begin{proof}
The statement is immediate for $\hat t \le \hat t_0$ by definition of $C$.
Note that the original flow remains uniformly strictly convex, and the curvature uniformly bounded on any compact subinterval of $[0,T)$, and so
the supremum $\sup_{\hat t \in [0,\hat t_0)}\{\hat E(\hat t)\}$ is finite.

It remains to deal with the case of $\hat{t}>\hat{t}_0$.
Here, we use Proposition \ref{prop entropy u negative}
to estimate the evolution of the entropy \eqref{entropy evolution1} by
\begin{equation}
	\label{entropyestimatedevo}
	2\pi\hat E_{\hat t} \leq  \int_0^{2\pi} 3\sqrt{2T}e^{-\hat{t}}\sigma_2\hat{k} \,d\hat{\theta}\,.
\end{equation}
Note that
\[
	\hat k e^{-\hat t} = \sqrt{2T} e^{-2\hat t}k
	= \sqrt{2T} e^{-2\hat t} (\sigma_1^{-1}F - \sigma_1^{-1}\sigma_2)
	\,.
\]
Integrating \eqref{entropyestimatedevo} and using the estimate~\eqref{kL1InTimeBoud}, we obtain
\begin{equation}\label{EntropyBoundRescaled}
\begin{aligned}
\hat E(\hat{t})-E(\hat{t}_0)
&\leq \frac{3T\sigma_2}{\sigma_1\pi}\int_{\hat{t}_0}^{\hat{t}}  \int_0^{2\pi} Fe^{-2\hat{t}} \,d\hat{\theta} d\hat{t} 
- \frac{3T\sigma_2^2}{\sigma_1\pi}\int_{\hat{t}_0}^{\hat{t}}  \int_0^{2\pi} e^{-2\hat{t}} \,d\hat{\theta} d\hat{t} 
\\
&\leq \frac{3\sigma_2 L(0)}{2\pi\sigma_1}
- \frac{3T\sigma_2^2}{\sigma_1}
\,.
\end{aligned}
\end{equation}
The proof is completed.
\end{proof}

The entropy bound gives many things, including a uniform estimate on rescaled length.

\begin{lem}
\label{rescaledlength}
Let $\hat{\gamma}:\S^1\times[0,\infty)\to \R^2$ be a solution to the rescaled
flow \eqref{mainFlowRescaled} with convex initial data.
There exists an $\hat L_0\in[0,\infty)$ depending only on the constant $C$ in
Proposition \ref{PRentropybddrescaled} and initial area of the rescaled
flow such that
\[
	\hat L(\hat t) \le \hat L_0
\]
for all $\hat t \in [0,\infty)$.
\end{lem}
\begin{proof}
Note that we have a uniform bound on the entropy
by Proposition \ref{PRentropybddrescaled}, and so Lemma
\ref{LemSupportFnAndWidth} provides a uniform positive lower bound on the
inradius of the rescaled flow, depending only on the constants in the entropy
bound.
As the area of the rescaled curves approach a constant (Lemma
\ref{LMareaconstrescaled}), the maximum diameter of the family of curves must
be bounded from above by a constant depending only on initial area and
$\sigma_1$.
Since the curve $\gamma$ (before rescaling) is uniformly convex up until the
final time (Theorem \ref{C-E thm curvature bound}), the rescaled curve
$\hat{\gamma}$ must also be convex.
As a result the rescaled length is bounded by $\pi$ times the diameter of
$\hat\gamma$, and the lemma follows.
\end{proof}

We finally prove a uniform estimate for rescaled curvature.

\begin{thm}
\label{TMrescaledcurvest}
Let $\hat{\gamma}:\S^1\times[0,\infty)\to \R^2$ be a solution to the rescaled
flow \eqref{mainFlowRescaled} with convex initial data.
There exists a $\hat k_1\in[0,\infty)$ such that
	\[
		\hat k_{\max}(\hat t) \le \hat k_1
	\]
for all $\hat t \in [0,\infty)$.
\end{thm}
\begin{proof}
We prove this by contradiction.

Recall the speed estimate~\eqref{FmaxPartialBound} in
Proposition~\ref{FmaxGrdientBound}. For a fixed time $t$ and a small
neighbourhood $|\theta-\theta^*|\leq 1/4\pi$ around where the spatial maxima
of curvature occurs, $\theta^*$ at $t$, we have
\[
\sigma_1 k_{\max}(t) < F_{\max}(t) \leq 2(\sigma_1k(\theta,t) +\sigma_2)+\frac{M}{2\pi}.
\]
Apply the rescaling to obtain
\[
\sigma_1\hat{k}_{\max}\left(\hat{t}\right) \leq 2\sigma_1
\hat{k}\big(\hat{\theta},\hat{t}\big) + \left(2\sigma_2
+\frac{M}{2\pi}\right)\sqrt{2T}e^{-\hat{t}},
\]
and rearrange to see
\[
	\hat{k}\big(\hat{\theta},\hat{t}\big) \geq \frac{1}{2}\left(\hat{k}_{\max}-C_1e^{-\hat t}\right),
\]
where
\(
C_1=\frac{1}{\sigma_1}\left(2\sigma_2 +\frac{M}{2\pi}\right)\sqrt{2T}\,.
\)
Thus
\[
	\log\hat{k}\big(\hat{\theta},\hat{t}\big) \geq \log\left(\hat{k}_{\max}\left(\hat{t}\right) - C_1e^{-\hat t} \right) -\log(2)
	\,.
\]
Integrating yields
\begin{equation}\label{C-E rescaled k bound cal1}
\int\displaylimits_{|\hat{\theta}-\hat{\theta}^*|\leq \frac{1}{4\pi}}\log\hat{k}\big(\hat{\theta},\hat{t}\big)\,d\hat{\theta} 
\geq \int\displaylimits_{|\hat{\theta}-\hat{\theta}^*|\leq \frac{1}{4\pi}}\log\left(\hat{k}_{\max}\left(\hat{t}\right) - C_1e^{-\hat t} \right) -\log(2) \,d\hat{\theta}
\,.
\end{equation}
Now suppose that $\hat{k}$ is unbounded.
The only way this can happen is asymptotically at infinity.
Thus there exists a sequence of times $\{\hat{t}_j\}$ with
$\hat{t}_j\rightarrow \infty$, such that
$\hat k_{\max}(\hat t_j) \rightarrow \infty$.
We may assume without loss of generality that $\hat t_1 > \hat t_0$ (where
$\hat t_0$ is as in Proposition \ref{prop entropy u negative}) and that $\hat
t_{j+1} > \hat t_j$.

For each $\hat{t}_j>\hat{t}_0$, the entropy
estimate \eqref{EntropyBoundRescaled} implies
\begin{equation}\label{C-E rescaled k bound cal2}
E\left(\hat{t}_0\right)
+ \frac{3\sigma_2 L(0)}{2\pi\sigma_1} - \frac{3T\sigma_2^2}{\sigma_1}
\geq E\left(\hat{t}_j\right) = \frac1{2\pi}\int_{\S^1} \log\hat{k}\big(\hat{\theta},\hat{t}_j\big)\,d\hat{\theta}.
\end{equation}
Let us fix $\hat t = \hat t_j$, and partition of the space domain in the following way:
\[
\S^1 = \left\{|\hat{\theta}-\hat{\theta}^*|\leq \frac{1}{4\pi}\right\} 
\cup 
\left\{\{\hat{k}<1\}\setminus\left\{|\hat{\theta}-\hat{\theta}^*|\leq \frac{1}{4\pi}\right\} \right\} 
\cup 
\left\{\{\hat{k}\geq 1\}\setminus\left\{|\hat{\theta}-\hat{\theta}^*|\leq \frac{1}{4\pi}\right\} \right\}\,.
\]
We estimate $E(\hat{t}_j)$ on each of these disjoint sets separately.
The estimate on the first set is given by~\eqref{C-E rescaled k bound cal1}.
On the second set, we note that $0>\hat{k}\log\hat{k}\ge -e^{-1}$ for
$\hat{k}\in(0,1)$, hence
\begin{equation}\label{C-E rescaled k bound cal3}
\begin{aligned}
\int\displaylimits_{\{\hat{k}<1\}\setminus\{|\hat{\theta}-\hat{\theta}^*|\leq \frac{1}{4\pi}\}} 
	\log\hat{k}\big(\hat{\theta},\hat{t}_j\big)\,d\hat{\theta}\, 
&\geq \int\displaylimits_{\{\hat{k}<1\}\setminus\{|\hat{\theta}-\hat{\theta}^*|\leq \frac{1}{4\pi}\}} \hat{k}\left(\hat{s},\hat{t}_j\right)
	\log\hat{k}\left(\hat{s},\hat{t}_j\right)\,d\hat{s}
\\
&\geq \int\displaylimits_{\{\hat{k}<1\}\setminus\{|\hat{\theta}-\hat{\theta}^*|\leq \frac{1}{4\pi}\}} -\frac{1}{e} \,d\hat{s}
\\
&\geq -\frac{1}{e}\hat{L}_{0}
\end{aligned}
\end{equation}
where $\hat{L}_{0}$ is the uniform bound on rescaled length from Lemma \ref{rescaledlength}.

Let us now consider the third set.
We note that on this set $\log \hat k \ge 0$, and so trivially
\[
\int\displaylimits_{\{\hat{k}\ge1\}\setminus\{|\hat{\theta}-\hat{\theta}^*|\leq \frac{1}{4\pi}\}} 
	\log\hat{k}\big(\hat{\theta},\hat{t}_j\big)\,d\hat{\theta} \ge 0\,.
\]
Combining \eqref{C-E rescaled k bound cal1}, \eqref{C-E rescaled k bound cal2},
\eqref{C-E rescaled k bound cal3} and the above, we obtain
\begin{align*}
E\left(\hat{t}_0\right) 
+ \frac{3\sigma_2 L(0)}{2\pi\sigma_1} - \frac{3T\sigma_2^2}{\sigma_1}
\geq \frac{1}{8\pi^2}\left[\log\left(\hat{k}_{\max}\left(\hat{t}_j\right) - C_1e^{-\hat t} \right)-\log(2)\right]-\frac{1}{2e\pi}\hat{L}_{0}
\,.
\end{align*}
While the left hand side is a finite constant, the right hand side tends to infinity
as $\hat{t}_j\longrightarrow\infty$, which is a contradiction. We deduce that
the normalised curvature $\hat{k}(\hat{\theta},\hat{t})$ must be uniformly
bounded from above.
\end{proof}

\section*{Acknowledgements}
The first author is supported by a IPRS Scholarship at University of Wollongong. The last two authors gratefully acknowledge the support of ARC grant DP150100375.

\bibliographystyle{plain}
\bibliography{CFWH}

\end{document}